\documentclass[11pt]{amsart}

\usepackage[all,cmtip]{xy}
\usepackage{amsmath, amssymb}
\usepackage{mathtools}
\usepackage[T1]{fontenc}
\usepackage[english]{babel}
\usepackage[hidelinks]{hyperref}
\usepackage[marginpar=2.5cm]{geometry}
\usepackage{xcolor}\usepackage{marginnote}
\usepackage{amsrefs}

\newtheorem{thm}{Theorem}[section]

\newtheorem*{thm*}{Theorem}
\newtheorem{lem}[thm]{Lemma}

\newtheorem{prop}[thm]{Proposition}
\newtheorem*{prop*}{Proposition}

\newtheorem{cor}[thm]{Corollary}

\theoremstyle{definition}
\newtheorem{defn}[thm]{Definition}
\newtheorem{notation}[thm]{Notation}
\newtheorem{remark}[thm]{Remark}
\newtheorem{question}[thm]{Question}
\newtheorem{problem}[thm]{Problem}
\newtheorem{example}[thm]{Example}
\newtheorem{exercise}[thm]{Exercise}

\newcommand{\dminus}{ 
\buildrel\textstyle\ .\over{\hbox{ 
\vrule height3pt depth0pt width0pt}{\smash-} 
}}

\def\e{\epsilon}
\def\la{\lambda}

\def\bb{\mathbb}

\def\vp{\varphi}
\def\de{\delta}

\def\Sg{\Sigma}

\def\bb{\mathbb}

\def\cc{\mathcal}

\DeclareMathOperator{\re}{Re}
\DeclareMathOperator{\im}{Im}
\DeclareMathOperator{\id}{id}
\DeclareMathOperator{\ad}{ad}

\DeclareMathOperator{\cp}{{CP}}

\DeclareMathOperator{\cb}{cb}

\DeclareMathOperator{\tr}{tr}

\DeclareMathOperator{\spn}{span}

\DeclareMathOperator{\ucp}{UCP}
\DeclareMathOperator{\MIN}{MIN}
\DeclareMathOperator{\MAX}{MAX}
\DeclareMathOperator{\dist}{dist}
\DeclareMathOperator{\domain}{domain}
\DeclareMathOperator{\codomain}{codomain}
\DeclareMathOperator{\Theory}{Theory}
\DeclareMathOperator{\Model}{Model}

\newcommand\ip[2]{\left\langle #1\, , #2 \right\rangle}

\DeclareMathOperator{\CP}{{CP}}

\begin{document}

\title{Model Theory of Operator Systems and C$^*$-Algebras}

\author[Sinclair]{Thomas Sinclair}

\address{Mathematics Department, Purdue University, 150 N. University Street, West Lafayette, IN 47907-2067}
\email{tsincla@purdue.edu}
\urladdr{http://www.math.purdue.edu/~tsincla/}

\subjclass[2020]{}

\keywords{}

\setcounter{tocdepth}{1}

\maketitle

\begin{abstract}
    We survey the model theory of operator systems and C$^*$-algebras.
\end{abstract}

\section{Introduction}

At the heart of the theory of operator algebras is the ``noncommutative order'' imposed on the algebra $\cc B(H)$ of bounded linear operators on some Hilbert space $H$ by the cone of positive semidefinite operators, that is, operators $T$ which admit a sum-of-squares decomposition $T = S_1^*S_1 + \dotsc S_n^*S_n$ for some $n$ and $S_1,\dotsc,S_n\in \cc B(H)$ (equivalently, $n=1$). This partial order structure, even in the finite-dimensional case, is incredibly rich and complex and its understanding is deeply connected with many important and outstanding open problems in fields as diverse as quantum information theory and quantum computing, numerical linear algebra and optimization, combinatorics, computer science, and random matrix theory. Beginning in the 1960s and 70s with work of Arveson, Choi, Effros, Lance, Kirchberg, and Stinespring it became apparent that the full power of the noncommutative order is captured not just by the ``level one'' order structure but the higher-level order structure imposed by taking matrix amplifications. In fact, the higher-order structure on a C$^*$-algebra is powerful enough to capture the norm-structure as well as many properties which would be considered algebraic in nature. This conception of noncommutative order crystallized in the notion of an \emph{operator system} first systematically studied in the groundbreaking work of Choi and Effros \cite{Choi1977}.

In any semisimple category of algebras the finite-dimensional, simple objects are of primary importance. In the case of C$^*$-algebras, these are the (complex) matrix algebras. Thus, it was a natural theme from the very beginning of the subject to try to understand the structure of a C$^*$-algebra by quantifying how well (or how poorly) it was algebraically approximated by direct sums of matrix algebras. With the theory of operator systems, it was realized that finite-dimensional \emph{order} approximation through matrix algebras was an equally vital aspect of the theory, and such properties as nuclearity and exactness have become standard tools of the trade for working operator algebraists. The surprising relations discovered by celebrated work of Kirchberg between two other well-studied finite-dimensional order approximation properties, the (local) lifting and weak expectation properties, and a famous conjecture of Connes have helped catalyze an entirely new chapter in operator algebras through their connection with a deep problem of Tsirelson arising in quantum information theory: see Goldbring's article in this volume.

The goal of this set of notes is to introduce the reader to some aspects of the continuous model theory of operator systems and C$^*$-algebras mainly through the model theoretic properties of the noncommutative order and their relation with finite-dimensional approximation properties.

This article is organized as follows.

\begin{itemize}
    \item Section 2 gives a brief overview of the theory of operator systems. While not totally self-contained, the section aims at presenting enough material for the reader to gain a basic working familiarity with operator systems, including such topics at the Choi--Effros representation theorem, Arveson's extension theorem, duality and the double dual, and quotients. The latter part of the section provides several examples of canonical ways of constructing operator systems. The reader may wish to consult the excellent books of Brown and Ozawa \cite{BrownOzawa} and Paulsen \cite{paulsen2002completely} for a fuller account of the theory. The reader may also wish to read the article of Szabo in this volume for an introduction to C$^*$-algebras before proceeding beyond this section.
    
    \item Section 3 introduces the model theory of operator systems and C$^*$-algebras. The first part of the section constructs the language for operator systems, while the next few sections serve to introduce basic model-theoretic concepts such as theories, ultraproducts, and definability, illustrating these ideas in the operator systems context before concluding with a discussion of the model theory of C$^*$-algebras. This section is largely based on the accounts found in the paper \cite{gs-kirchberg} of Goldbring and the author and the monograph of Farah, et al.\ \cite{model-c-star}, though many of the proofs of results given here are new and the centering of the account from the perspective of matrix completion problems is also somewhat novel. A reader wishing to explore these topics further would do good to consult \cite{model-metric} and \cite{model-c-star}.
    
    \item Section 4 surveys some applications of the model theory developed in the Section 3 to the theory of finite-dimensional approximation properties for operator systems and C$^*$-algebras. The discussion of exactness and nuclearity is largely sourced from \cite{model-c-star} and \cite{gs-omitting}, while that of the lifting property is from \cite{gs-kirchberg}, \cite{gs-omitting}, and \cite{sinclair-cp}. The last part of this section on the weak expectation property is sourced from \cite{gs-omitting} and \cite{lupini-wep2018}, though the treatment here is substantially new. We refer the reader to Goldbring's article in this volume for the model-theoretic aspects of the local lifting and weak expectation properties in relation to the famous conjectures of Connes and Kirchberg, and the reader may also wish to consult the recent monograph of Pisier \cite{pisier-ck} on  this subject.
\end{itemize}

\section{General Background on Operator Systems} 

\begin{defn}\label{defn:os-1}
    An \emph{operator system} $E$ is a closed subspace of $\cc B(H)$ which is closed under taking adjoints and contains the unit. 
\end{defn}

\begin{remark} \label{rmk:pos-cone}
    Let $E^h$ be the set of hermitian elements of $E$. From the property that $E$ is closed under adjoints, it is immediate that $E^h$ is a real, unital subspace with $E = \bb C E^h$. Since $E$ contains the unit, we have that for each $x\in E^h$ we may decompose $x$ as a difference of positive elements in $E$, namely as $x = (\|x\|1 + x) -  \|x\|1$, which shows that $E^+ := E\cap \cc B(H)^+$ is a cone in $E$.
    In this way $E^h$ is an ordered (real) vector space under $x\preceq y \longleftrightarrow y-x\in E^+$. Moreover $1$ is an \emph{order unit}, that is, for each $x\in E^h$ $-r1\preceq x \preceq r1$ for some $r>0$ which is \emph{archimedean} in the sense that $-\e1\preceq x\preceq \e1$ for all $\e>0$ implies that $x=0$.
\end{remark}

So far we have just established that an operator system is a Banach space with a pleasant real, ordered structure. In the real case, these are referred to as \emph{function systems} and were studied and classified by Kadison \cite{kadison} (see \cite[Chapter II]{alfsen} as well), while a study of the complex case was undertaken by Paulsen and Tomforde \cite{paulsen-tomforde}. What distinguishes operator systems from function systems is the following enrichment of structure of the objects or, to put it in a glass-half-empty way, restriction of the category maps between the objects.

The enrichment comes from the following, seemingly modest, observation: if $E\subset \cc B(H)$ is an operator system, then so is $E_n := M_n(E)\subset M_n(\cc B(H))\cong \cc B(H^{\oplus n})$, where $H^{\oplus n}$ is the direct sum of $n$ copies of $H$. Thus, from each operator system there may be derived a sequence of order unit spaces $(E_n, E_n^+, 1_n)$. Here, and throughout, $1_n$ is the tensor of the identity matrix $I_n$ in $M_n$ with the unit $1$ in $E$. Moreover, there is a natural family of connecting maps defined as follows. For $a\in M_{n,k}$ define $\ad(a): E_n \to E_k$ by $\ad(a): x\mapsto a^*xa$. Notice that $\ad(a)$ preserves the order structure. For $v\in M_{n,k}$ with $v^*v=1_k$ we will refer to $\ad(v)$ as the \emph{compression} induced by $v$. (Notice that $vv^*$ is a rank $k$ projection in $M_n$, so $n\geq k$.)

Given a $\ast$-linear map $\vp: E\to F$ between operator systems $E$ and $F$, we can define a $\ast$-linear map $\vp_n: E_n\to F_n$ coordinate-wise by $\vp_n([x_{ij}]) := [\vp(x_{ij})]$. Thinking about $M_n(E)\cong M_n\otimes E$, $\vp_n$ is nothing other than $\id_{M_n}\otimes\, \vp$. Note that the $\vp_n$'s commute with the connecting maps (thus, compressions): $\ad(a)\circ \vp_n = \vp_k\circ \ad(a)$.

\begin{defn}
    A map $\vp: E\to F$ is said to be \emph{$n$-positive} ($n=1,2,3,\dotsc$) if $\vp_n(E_n^+)\subset F_n^+$ and \emph{completely positive} if $\vp$ is $n$-positive for all $n$. We say that $\vp$ is a \emph{complete order embedding} if it is unital and both $\vp$ and $\vp^{-1}: \vp(E)\to E$ are completely positive.
\end{defn}

\noindent It is easy to see that $n$-positivity implies $k$-positivity for all $k\leq n$ and that positive maps are $\ast$-linear.

For an operator system $E\subset \cc B(H)$, we denote by $\|\cdot\|_n$ the restriction of the operator norm on $\cc B(H^{\oplus n})$ to $E_n$.

\begin{defn}
    A linear map $\vp: E\to F$ is said to be \emph{$n$-bounded} if $\vp_n: (E_n,\|\cdot\|_n)\to (F_n,\|\cdot\|_n)$ is bounded, in which case we denote the norm of $\vp_n$ by $\|\vp\|_n$. The map $\vp$ is said to be \emph{completely bounded} if $\sup_n\|\vp\|_n<\infty$, in which case we write $\|\vp\|_{\cb} := \sup_n \|\vp\|_n$. (Note that $\|\vp\|_k\leq \|\vp\|_n$ for all $k\leq n$.)
\end{defn}

Perhaps the most important foundational fact in the the theory of operator systems is that the higher-order norm structure is totally determined by the structure of the positive cones.

\begin{prop} \label{prop:order-norm}
     Let $E\subset \cc B(H)$ be an operator system. For $x\in E_n$ we have that
     \[\|x\|_n = \inf\left\{t>0 : \begin{bmatrix} t1 & x\\ x^* & t1\end{bmatrix}\in E_{2n}^+\right\}.\]
\end{prop}

\begin{proof}
    It clearly suffices to check the case $n=1$. We have that $\begin{bmatrix} a & b\\ b^* & c\end{bmatrix}\in M_2(\cc B(H))^+$ if and only if \[\ip{a\xi}{\xi} + 2\re\ip{b\xi}{\eta} + \ip{c\eta}{\eta}\geq 0\] for all $\xi,\eta\in H$ with $\|\xi\|^2+\|\eta\|^2=1$. It follows that $\begin{bmatrix} t1 & x\\ x^* & t1\end{bmatrix}$ is positive if and only if $2t + 2\re\ip{x\xi}{\eta}\geq 0$ for all unit-norm vectors $\xi,\eta\in H$. This, in turn, is seen to be equivalent to $t\geq |\ip{x\xi}{\eta}|$ for all such $\xi,\eta$ by making substitutions of the form $\eta\mapsto e^{is}\eta$ for $s\in \bb R$ chosen suitably. Since we have that \[\|x\| = \sup\{|\ip{x\xi}{\eta}| : \|\xi\|=\|\eta\|=1\},\] this completes the proof. \qedhere
\end{proof}

\begin{exercise} \label{ex:unital-2-positive-contraction}
    A unital, $2$-positive map $\vp: E\to F$ is contractive.
\end{exercise}

\begin{prop} \label{prop:contraction-positive}
    If $\vp: E\to F$ is a unital, $\ast$-linear contractive map between operator systems, then $\vp$ is positive.
\end{prop}

\begin{proof}
    Let $x\in E^+$ with $0\preceq x\preceq 1$, and define $y = 2x -1$ so that $-1\preceq y\preceq 1$, equivalently $\|y\|\leq 1$. Since $\vp$ is unital and $\ast$-linear $\vp(y) = 2\vp(x) - 1$ is hermitian, and $\|\vp(y)\|\leq 1$ since $\vp$ is contractive. Thus, $-1\preceq 2\vp(x) -1\preceq 1$, which implies $\vp(x)\succeq 0$.
\end{proof}

\begin{prop} \label{ex:contraction-positive}
    Any unital, contractive linear map between operator systems is $\ast$-linear.
\end{prop}

\begin{proof}
    Let $\vp: E\to F$ be unital and contractive with $E\subset \cc B(H)$ and $F\subset \cc B(K)$. Let us fix a unit vector $\xi\in K$ and define $\phi_\xi(x) := \ip{\phi(x)\xi}{\xi}$. For $x\in E^h$ a contraction and $t\in \bb R$, using the identity $\|z^*z\| = \|z\|^2$ for all $z\in \cc B(K)$ we have
    \[|\phi_\xi(x) + it1|\leq \|x + it1\| = \|(x +it1)^*(x + it1)\|^{1/2} = \|x^2 + t^21\|^{1/2}\leq (1+t^2)^{1/2},\]
    since $x^2+t^21\preceq (\|x\|^2 + t^2)1$. It is now easy to see that 
    \[|\im \vp_\xi(x)|\leq (1-t^2)^{1/2} - |t|\to 0\ \textup{as}\ t\to\pm \infty.\]
    Since $\vp_\xi(x)$ is then real for all $x\in E^h$, we have that $\vp_\xi$ is $\ast$-linear. It follows that $\vp$ is $\ast$-linear since for a (complex) Hilbert space $K$ an operator $z\in \cc B(K)$ is hermitian if and only if $\ip{z\xi}{\xi}\in\bb R$ for all $\xi\in K$ \cite[Proposition 2.12]{Conway}. \qedhere

\end{proof}

We derive the following easily as a consequence of the previous two results.

\begin{cor} \label{cor:isom-order}
    Every complete order embedding of operator systems $\vp: E\to F$ is a \emph{completely isometric embedding}, that is, $\vp_n: E_n\to F_n$ and $(\vp^{-1})_n: \vp(E)_n\to E_n$ are isometries for all $n=1,2,\dotsc$. Every unital completely isometric embedding is a complete order embedding.
\end{cor}

\begin{exercise}[Kadison--Cauchy--Schwarz]
    Let $E\subset \cc B(H)$ be an operator system. If $\vp: E\to \cc B(K)$ is unital and $2$-positive, then $\vp(x)\vp(x^*)\preceq \vp(xx^*)$ for all $x\in E$. See \cite[Chapter 3]{paulsen2002completely} for this and many other basic properties satisfied by completely positive maps.
\end{exercise}

\begin{exercise}
    If $\vp: E\to F$ is completely positive, then $\|\vp\|_{\cb} = \|\vp(1)\|$. See \cite[Proposition 3.6]{paulsen2002completely}.
\end{exercise}

\begin{prop}[Choi's Theorem \cite{Choi1975}] \label{ex:chois-thm}
    Every completely positive map $\vp: M_n\to M_k$ is of the form $\vp(x) = \sum_{i=1}^{nk} a_i^*xa_i$ for some $a_1,\dots,a_{nk}\in M_{n,k}$. 
\end{prop}

The converse is also seen to be true as $x\mapsto a^*xa$ is completely positive. See \cite[Theorem 2.21]{aubrun} for a proof. The matrices $a_i$ above are said to form a \emph{Kraus decomposition} of the map $\vp$.

\begin{exercise} \label{ex:cp-adjoint}
      Use Choi's Theorem to deduce that if $\vp: M_n\to M_k$ is completely positive, then the \emph{adjoint} map $\vp^\dagger: M_k\to M_n$ defined by the functional equation $\tr_k(\vp(A)B) = \tr_n(A\vp^\dagger(B))$ for all $A\in M_n$, $B\in M_k$ is again completely positive. If $\vp$ is unital, then $\vp^\dagger$ is trace-preserving, $\tr_k\circ\, \vp = \tr_n$, and vice versa.
\end{exercise}

\begin{exercise}
    For $E = M_2$ show that 
    \[\vp: \begin{bmatrix} a & b\\ c & d\end{bmatrix}\mapsto \begin{bmatrix} (a+d)/2 & b\\ c & (a+d)/2\end{bmatrix}\] is unital and $1$-positive, but not $2$-positive. Show that $|\la|= 1/2$ is the optimal constant so that 
    \[\vp: \begin{bmatrix} a & b\\ c & d\end{bmatrix}\mapsto \begin{bmatrix} (a+d)/2 & \la b\\ \la c & (a+d)/2\end{bmatrix}\] is completely positive.
\end{exercise}

\begin{exercise}
    Let $v_1,v_2$ be the generators of the Cuntz algebra $\cc O_2$ (see Szabo's article in this volume), and consider $E := \spn\{1, v_1, v_2, v_1^*,v_2^*\}$. Show that $\vp: E\to E$ given by $\vp(a1 + bv_1 + cv_2 + dv_1^* + ev_2^*) = a1 + bv_1^* + cv_2^* + dv_1 + ev_2$ is unital and positive, but not $2$-positive, as $\vp_2$ is not contractive.
\end{exercise}

\begin{exercise}
    We have that $\vp: M_2\to M_2$ given by transposition, $\vp(x) := x^t$, is positive but not $2$-positive.
\end{exercise}

Remarkably, the structure of higher ordered cones and connecting maps is totally sufficient to abstractly characterize operator systems, as discovered in the seminal work of Choi and Effros \cite{Choi1977}. We will say that $(E,E^+,e)$ is a \emph{real ordered $\ast$-vector space} if $E$ is a $\ast$-vector space with a cone $E^+\subset E^h$ and order unit $e\in E^+$. 

\begin{defn} \label{defn:abstract-os}
    An \emph{abstract operator system} $(E,E^+,1)$ is a real ordered $\ast$-vector space with archimedean order unit $1$ along with a real ordered $\ast$-structure $(E_n, E_n^+,1_n)$ for each $n=1,2,\dotsc$, where $E_n\cong M_n\otimes E$, $(a\otimes x)^* = a^*\otimes x^*$, and $1_n = I_n\otimes 1$, so that $E_n^+\subset E_n^h$ and $a^*E_n^+a\subset E_k^+$ for all $a\in M_{n,k}$.  
\end{defn}

\begin{remark}
    Note that the condition $a^*E_n^+a\subset E_k^+$ for $n=1$ yields that $M_k^+\otimes E^+\subset E_k^+$ for all $k$ by the spectral theorem.
\end{remark}

\begin{remark} \label{rem:matrix-ordered-ast-vector-space}
    Omitting the requirement of an order unit in the previous definition leads to the definition of a \emph{matrix ordered $\ast$-vector space}
\end{remark}

\begin{thm}[Choi and Effros] \label{thm:choi-effros}
    Every abstract operator system $E$ admits a complete order embedding into $\cc B(H)$ for some Hilbert space $H$.
\end{thm}

The representation obtained by the Choi--Effros theorem is canonical, but, like the Gelfand--Naimark theorem, it is impractical to work with. Nonetheless, the abstract characterization of operator provides us with a way to axiomatize the class of operator systems as metric structures. Additionally, the proof of this theorem introduces some important ideas, as we will see in the following sketch. 

To begin, we introduce the notion of a \emph{state} $\vp: E\to \bb C$ to just mean that $\vp$ is a unital, positive map. It follows essentially by an application of the Hahn--Banach theorem for archimedean ordered $\ast$-vector spaces that the states separate points in $E$ and completely determine the positivity structure;  to wit,

\begin{lem}\label{lem:state-positive}
    For an abstract operator system $E$ and $x\in E$, we have that $x\in E^+$ if and only if $\vp(x)\geq 0$ for every state $\vp$.
\end{lem}

See \cite[Proposition 3.12]{paulsen-tomforde} for a proof.

\begin{remark} \label{rem:archimedean}
    Even if $(E, E^+, e)$ is a  real ordered $\ast$-vector space with $e$ being an not necessarily archimedean order unit, one may still consider the set of all states on $E$ and define $E'$ to be their closed, linear span in $E^*$. For each $x\in E$ we can define the \emph{evaluation map} $\hat x: E'\to \bb C$ by $\hat x(\vp) := \vp(x)$ and set $\widehat E := \{\hat x : x\in E\}$ to be the image of $E$ under this map. It can be seen that $(\widehat E, \widehat{E^+}, \hat e)$ is now an archimedean ordered, $\ast$-vector space which is known as the \emph{archimedeanization} of $E$. We refer the reader to \cite[Section 3.2]{paulsen-tomforde} for details. 
    
    The crucial point is that $e$ is an archimedean order unit \emph{exactly when} the evaluation map $x\mapsto \hat x$ is faithful, that is, has trivial kernel. In many of the constructions going forward we will define some sort of higher-level positivity/order structure on a archimedean ordered $\ast$-vector space, making it an operator system, by describing the set of higher-order positive maps. At a technical level the positive maps really induce an operator system structure on $\widehat E$ which we can identify with $E$ canonically as long as we know the evaluation map is faithful.
\end{remark}

\begin{lem}\label{lem:state-cp}
    For an abstract operator system $E$, every positive map $\vp: E\to \bb C$ is completely positive. In particular, every state is completely positive.
\end{lem}

\begin{proof}
    Let $x\in E_n^+$. We want to show that $\vp_n(x)\in M_n$ is positive semidefinite. To this end, choose $a\in M_{n,1}$, and observe that $a^*\vp_n(x)a = \vp(a^*xa)\geq 0$ since $a^*xa\in E^+$. We are therefore done since $a^*\vp_n(x)a\geq 0$ for all $a\in M_{n,1}$ characterizes the matrix $\vp_n(x)$ as being positive semidefinite. \qedhere
\end{proof}

\begin{notation}
    For (abstract) operator systems $E$ and $F$, we use $\cp(E,F)$ and $\ucp(E,F)$, respectively, to denote the set of all completely positive maps from $E$ to $F$ and the set of all unital, completely positive maps from $E$ to $F$.
\end{notation}

\begin{lem} \label{lem:cp-to-ucp}
    Let $E$ be an abstract operator system. For every $\vp\in \cp(E,M_n)$ there is $\vp'\in \ucp(E,M_k)$ for some $k\leq n$ so that $\ker(\vp') = \ker(\vp)$.
\end{lem}

\begin{proof}
    If $\vp$ is unital, then there is nothing to prove. Let $z := \vp(1)\in M_n^+$ and let $p$ be its support projection, that is, the largest projection $p$ so that $pxp$ is invertible $pM_n p$. Let $x\in E$. Since $1_2$ is an order unit in $E_2$ we may assume without loss of generality that $\begin{bmatrix} 1 & x\\ x^* & 1\end{bmatrix}\succeq 0$, so $\begin{bmatrix} z & \vp(x)\\ \vp(x)^* & z\end{bmatrix} \succeq 0$ as well. It follows that 
    \[\begin{bmatrix} 0 & \vp(x)p^\perp\\ p^\perp\vp(x^*) & z\end{bmatrix} = \begin{bmatrix} p^\perp z p^\perp & \vp(x)p^\perp\\ p^\perp\vp(x^*) & z\end{bmatrix} \succeq 0,\]
    hence $\vp(x)p^{\perp} = 0$  for all $x\in E$ and $p^\perp \vp(x)=0$ as well by $\ast$-linearity of $\vp$. If $p$ is of rank $k$, we find $a\in M_{n,k}$ so that $pa = a$ and $a^*za = I_k\in M_k$. Setting $\vp'(x) := a^*\vp(x)a = a^*p\vp(x)pa\in \ucp(E,M_k)$, it is clear that $\vp'(x) = 0$ if and only if $\vp(x)=0$.
\end{proof}

\begin{lem}\label{lem:adjoint}
    Let $E$ and $F$ be abstract operator systems. There is an affine bijection between $\cp(E_n, F)$ and $\cp(E,F_n)$, $\cp(E_n,F)\ni\vp\longleftrightarrow \tilde\vp\in \cp(E,F_n)$, given by
    \[ \tilde\vp(x) := [\vp(e_{ij}\otimes x)]_{ij}, \quad \vp([x_{ij}]) := \sum_{i,j}\tilde\vp(x_{ij})_{ij},\]
    where $e_{ij}$ are the standard matrix units for $M_n$.
\end{lem}

\begin{proof}
    Let's start with the assumption that $\vp: E_n\to F$ is completely positive. It is clear that the map $\Delta_n: M_n\to M_n\otimes M_n$ given by \[\Delta_n([x_{ij}]) := \sum_{i,j} e_{ij}\otimes x_{ij}e_{ij}\]
    is a (non-unital) $\ast$-homomorphism, hence is completely positive. Therefore by Proposition \ref{ex:chois-thm}, we have that $\Delta_n(x) = \sum_i \ad(v_i)(x)$, $v_i\in M_{n,n^2}$; hence, for all $x\in E_n^+$ we have $(\Delta_n\otimes \id_E)(x) = \sum_i (\ad(v_i)\otimes \id_E)(x)\in E_{n^2}^+$. Since $\vp$ is completely positive, so is $\id_{M_n}\otimes\,\vp: M_n\otimes E_n\to F_n$. Let $j_n\in M_{n,1}$ be the matrix of all $1$'s. We have that the connecting map $\ad(j_n^*): E\to E_n$ as defined at the beginning of this section is completely positive. We now compute that
    \[\tilde\vp(x) = \left((\id_{M_n}\otimes\, \vp_n)\circ (\Delta_n\otimes \id_E)\circ \ad(j_n^*)\right)(x),\] 
    so $\tilde\vp: E\to F_n$ is completely positive.
    
    We now turn to the case that $\tilde\vp: E\to F_n$ is completely positive, so $\id_{M_n}\otimes\, \tilde\vp: E_n\to M_n\otimes F_n$ is as well. We have that
    \[\vp([x_{ij}]) = \ad(j_n)\circ (\Delta_n^\dagger\otimes\id_F)\circ(\id_{M_n}\otimes\, \tilde\vp)([x_{ij}])\]
    is completely positive by Exercise \ref{ex:cp-adjoint}. \qedhere
    
\end{proof}

We now sketch a proof of Theorem \ref{thm:choi-effros}.

\begin{proof}[Sketch of a Proof of Theorem \ref{thm:choi-effros}]
    Consider all $\vp\in \ucp(E,M_n)$, over all $n=1,2,\dotsc$, where we will write $n(\vp) := n$ for clarity. By Lemma \ref{lem:adjoint} we have that $\cp(E,M_n)\cong \cp(E_n,\bb C)$, which by Lemmas \ref{lem:cp-to-ucp}, \ref{lem:state-cp}, and \ref{lem:state-positive} shows that \[\bigcup_{1\leq k\leq n}\ucp(E,M_k)\] completely determines the positive cone in $E_n$. For concision let's denote \[I := \bigcup_{n=1}^\infty \ucp(E,M_n).\] 
    
    It is now relatively straightforward to check that setting $H = \bigoplus_{\vp\in I} \bb C^{n(\vp)}$ we have that $\Phi: E\to \cc B(H)$ given by 
    \[\Phi(x) = \prod_{\vp\in I} \vp(x)\in \prod_{\vp\in I} M_{n(\vp)} \subset \cc B(H)\]
    is a complete order embedding. \qedhere
\end{proof}

\begin{remark}\label{rmk:op-subsys-matrix-product}
    It is useful to note that the proof shows that every operator system is a subsystem of a direct product of matrix algebras.
\end{remark}

%

One of the main upshots of the Choi--Effros representation theorem from our vantage is that it describes the following ``dual'' way of viewing an operator system structure via a coherent collection of admissible positive maps into each matrix algebra. Let $(E,E^+,1)$ be a real ordered $\ast$-vector space with archimedean order unit, and let $\cc S(E)$ denote the set of states of $E$.

\begin{prop} \label{prop:dual-os}
    Suppose that for each $n=1,2,\dotsc$ we define a closed set $\cc S_n$ of unital, linear maps from $E$ to $M_n$ so that $\cc S_1 = \cc S(E)$ and for each $f\in \cc S_n$ and $\vp\in \ucp(M_n,M_k)$  we have that $\vp\circ f\in \cc S_k$. Then the collection $(\cc S_n)$ completely determines an operator system structure on $E$ so that $\cc S_n = \ucp(E,M_n)$ for all $n$.
\end{prop}

\begin{proof}[Sketch of Proof]
    For each $\vp\in \cc S_n\subset \cc L(E, M_n)$, write $\vp(x) = [\vp_{ij}(x)]$. Define $E_n^+\subset E_n$ to be all $x\in E_n$ so that $\sum_{i,j} \vp_{ij}(x_{ij})\geq 0$ for all $\vp\in \cc S_n$. We leave it as an exercise that this defines a compatible family of real ordered structures on each $E_n$ satisfying Definition \ref{defn:abstract-os}.
\end{proof}

Perhaps the most important basic result in the theory of operator systems is Arveson's extension theorem.

\begin{thm}[Arveson's Extension Theorem] \label{thm:arveson}
    For every unital inclusion $E\subset F$ of operator systems and every unital, completely positive map $\vp: E\to \cc B(H)$, there exists a unital, completely positive map $\tilde\vp: F\to \cc B(H)$ extending $\vp$.
\end{thm}

We will outline a proof here: for a detailed proof see \cite[Theorem 1.6.1]{BrownOzawa} or \cite[Theorem 7.5]{paulsen2002completely}

\begin{proof}[Sketch of Proof]
    We first note that any positive map $\vp: E\to \bb C$ extends to a positive (hence, completely positive) map $\tilde\vp: F\to \bb C$ by the Hahn--Banach extension theorem for ordered vector spaces \cite[Corollary 9.12]{Conway}. We now consider the case when $\dim(H) = n<\infty$. We have that $\CP(E,\cc B(H))\cong \CP(E,M_n)\cong \CP(M_n(E),\bb C)$. As before, we can now extend any $\vp\in \CP(M_n(E),\bb C)$ to $\tilde\vp\in \CP(M_n(F),\bb C)\cong \CP(F,M_n)$. 
    
    Now let $H$ be arbitrary and let $(H_i)_{i\in I}$ be the directed set of all finite-dimensional subspaces of $H$, ordered by inclusion. Let $p_i: H\to H_i$ be the orthogonal projection onto $H_i$ and consider the unital completely positive maps $\vp_i: E\to \cc B(H_i)$ defined by $\vp_i(x) = p_i\vp(x)p_i$. We may extend each to $\tilde\vp_i: F\to \cc B(H_i)$ unitally and completely positively. Since all maps are contractive, we can now take a pointwise-ultraweak cluster point of the net $\tilde\vp_i$ to obtain a unital, completely positive map $\tilde\vp: F\to \cc B(H)$. It is straightforward to check that $\tilde\vp$ extends $\vp$.
\end{proof}

We now turn our attention to describing the dual of an operator system. We begin with the following proposition. 

\begin{lem}
    Let $E$ be an operator system. Every bounded linear functional $\vp: E\to \bb C$ is a linear combination of four states. Every completely bounded map $\vp: E\to M_n$ is a linear combination of four completely positive maps.
\end{lem}

\begin{proof}
    Suppose $E\subset \cc B(H)$. By the Hanh-Banach theorem every bounded linear contraction $\vp: E\to \bb C$ extends to a linear contraction $\tilde\vp: \cc B(H)\to \bb C$. Now $\tilde\vp$ is a linear combination of four states by the Jordan decomposition theorem for C$^*$-algebras \cite[Proposition 2.1]{takesaki-i}. The rest of the result follows from this by essentially the same proof as Lemma \ref{lem:adjoint} by noting the correspondence described therein converts any completely bounded linear map $\vp: E\to M_n$ to a bounded linear functional $\tilde\vp: M_n(E)\to \bb C$.
\end{proof}

From this we see that $E^*$ is equipped with a matricial ordered $\ast$-vector space structure, where the positive cone in $M_n(E^*)$ is identified with the completely positive maps $E\to M_n$. This does not give $E^*$ the structure of an operator system, however, as there is no clear choice of an archimedean order unit. In the case that $E$ is a finite-dimensional operator system there is a way of placing an operator system structure on $E^*$ which is non-canonical as it will depend on choosing an order unit. This is essentially a consequence of the following lemma: see \cite[Lemma 2.5]{Kavruk2014} for a proof.

\begin{lem}
    For any finite-dimensional operator system $E$ we can find a basis $\{1=x_1,\dotsc,x_n\}$ of self-adjoint elements so that the dual basis $\{x_1^*,\dotsc,x_n^*\}$, defined by $x_i^*(x_j) = \de_{ij}$, is a hermitian basis for $E^*$ with $x_1^*$ an order unit.
\end{lem}

A related, useful lemma appears as \cite[Lemma B.10]{BrownOzawa}.

\begin{lem} \label{lem:os-dual-basis}
    For any finite-dimensional operator system $E$, we can find a spanning set $1= x_1,\dotsc, x_n$ consisting of hermitian elements of unit norm so that the dual basis $x_1^*,\dotsc,x_n^*$ defined by $x_i^*(x_j) = \delta_{ij}$ satisfies $\|x_i^*\| =1$ as well for all $i=1,\dotsc,n$. Moreover, $x_1^*$ may be taken to be an order unit in the cone $\cc S(E)$.
\end{lem}

Even though $E^*$ is not an operator system, $E^{**}$ is, with the matricial order structure obtained by dualizing twice being given by the weak*-closures of the positive cones $E_n^+$ under the evaluation map $\hat{\cdot}$. The tricky part is checking that $\hat 1: E^*\to \bb C$ is an \emph{archimedean} order unit for this matricial ordered structure on $E^{**}$. 
In fact, the following is true.

\begin{prop} \label{prop:E-dual-dual}
    The natural embedding $\hat{\cdot}: E\hookrightarrow E^{**}$ given by sending $x\mapsto \hat x$ is a complete order embedding of operator systems. If $E$ is finite-dimensional, then it is a complete order isomorphism.
\end{prop}
 
\begin{proof}
    We begin by showing that $\cc B(H)^{**}$ has the structure of a unital C$^*$-algebra whose positive cone agrees with that of the induced matricial order structure and whose unit is $\hat 1$. It follows that $\hat 1$ is an archimedean order unit in $\cc B(H)^{**}$ since C$^\ast$-algebras are operator systems. (The author is not aware of any proof of this proposition which avoids this reasoning.)  
    Indeed, by \cite[Section III.2]{takesaki-i} there is a unital $\ast$-algebra embedding $\pi: \cc B(H)\to \cc B(K)$ so that there is a unique extension $\pi': \cc B(H)^{**}\to \cc B(K)$ which is a homeomorphism from the weak*-topology on $\cc B(H)^{**}$ to the ultraweak topology on $\cc B(K)$. This implies that $x\in M_n(\cc B(H)^{**})$ is positive in the matricial order structure defined on $\cc B(H)^{**}$ if and only if $\pi'(x)$ is positive in $M_n(\cc B(K))$ as both are determined from taking the closures of $M_n(\cc B(H))^+$ in the respective weak topologies.

    Now, we have that $E\subset \cc B(H)$ for some Hilbert space $H$, and it is straightforward to see that $E^{**}\hookrightarrow \cc B(H)^{**}$ is a complete order embedding at the level of the induced matricial ordered structures. Since $1 = \pi(1)\in \cc B(K)$ is archimedean, we have that $\hat 1\in E^{**}$ is archimedean, thus $E^{**}$ is an operator system.
\end{proof}

\begin{defn}
    For an operator system $E$, a subspace $J\subset E$ is said to be a \emph{kernel} if $J = \bigcap_{i\in I} \ker(\vp_i)$ where $\{\vp_i : i\in I\}$ is a collection of states on $E$.
\end{defn}

Given a kernel $J\subset E$, we may put a \emph{quotient operator system} structure on $\widehat{E/J}$ for the ordered $\ast$-vector space quotient $E/J$ by identifying $\ucp(E/J,M_n)$ with the weak$^*$-closed, convex subset $\ucp(E,M_n; J)$ of $\ucp(E,M_n)$ consisting of all $\vp$ with $J\subset \ker(\vp)$. Note that in general the quotient structure must be placed on $\widehat{E/J}$ and not $E/J$ itself as the class of $1$ in $E/J$ is an order unit which is possibly not archimedean, so we must pass through the evaluation map to obtain the archimedeanization. Quotients of operator systems were first defined by Choi and Effros \cite{Choi1977}. We refer the reader to \cite{Kavruk2014, kptt2013} for a in-depth treatment of quotients of operator systems.

\subsection{Examples of Operator Systems}

\begin{example}
    Every unital C$^*$-algebra $A$ is canonically an operator system where $A_n^+$ is just the cone of positive elements in the C$^*$-algebra $M_n(A)$. In particular the matrix algebra $M_k$ is naturally an operator system under $(M_k)_n^+ = (M_n\otimes M_k)^+$. Note that not every element of $(M_n\otimes M_k)^+$ can be written as a sum of simple tensors $x\otimes y$ with $x\in M_n^+$ and $y\in M_k^+$ as soon as $n,k\geq 2$! In quantum information theory the elements in $(M_n\otimes M_k)^+$ which are not in the span of simple tensors of positive elements are said to be \emph{entangled}.
\end{example}

\begin{example}
    Let $X$ be a finite set. For a vector space $V$, let $V[X]$ denote the vector space of all functions $f: X\to V$. We have that $\bb C[X]$ is a real, ordered $\ast$-vector space under pointwise conjugation and the cone $\bb R_{\geq 0}[X]$ of all non-negatively-valued functions $f: X\to \bb R_{\geq 0}$, with order unit being the characteristic function $1_X$.
     
    We can define the following ``free'' canonical operator system structure on $\bb C[X]$ as follows. We have that $M_n(\bb C[X]) \cong M_n[X]$ canonically, so we can define $\bb C[X]_n^+$ as the cone of all functions $f: X\to M_n^+$. In the dual picture, consider a map $g: X\to M_k^+$, which extends to a positive linear map $\tilde g: \bb C[X]\to M_k$. We see that $\tilde g_n: M_n[X]\to M_n(M_k)\cong M_n\otimes M_k$ is given by $\tilde g_n(f) = \sum_{x\in X} f(x)\otimes g(x)$; hence, every positive map from $\bb C[X]$ to $M_k$ is completely positive.  On the other hand, consider a positive map $g: X\to M_k^+$. For $f\in M_n[X]$ we have the pairing $\sum_{i,j} \tilde g_{ij}(f_{ij})$ as given in the proof of Proposition \ref{prop:dual-os} can be computed as
    \[ \sum_{i,j} \tilde g_{ij}(f_{ij}) =\sum_{x\in X} \tr(f(x)g(x)).\]
    Since $A\in M_n$ is positive semidefinite if and only if $\tr(AB)\geq 0$ for all $B$ positive semidefinite, it follows that $\sum_{i,j} \tilde g_{ij}(f_{ij})\geq 0$ for all $g: X\to M_n^+$ if and only if $f(x)\in M_n^+$ for all $x\in X$. These arguments are seen to apply equally well replacing $M_n$ with an arbitrary direct product of matrix algebras as the direct sum is dense in the strong operator topology; thus, the reasoning applies to \emph{all} operator systems by Remark \ref{rmk:op-subsys-matrix-product}.
    
    In summary, we see that this operator system structure on $\bb C[X]$ has the following two universal properties: for every operator system $E$ every positive map $\vp: \bb C[X]\to E$ is completely positive and every function $f:X\to E^+$ induces a completely positive map $\tilde f: \bb C[X]\to E$. Let us also observe that a map $\psi: E\to \bb C[X]$ is positive if and only if $\de_x\circ \psi: E\to \bb C$ is positive for all $x\in X$, where $\de_x$ is the point evaluation map. Thus for any operator system $E$ any positive map $\psi: E\to \bb C[X]$ is completely positive.
\end{example}

\begin{example} \label{ex:linear-quotient}
    We now describe the construction of universal operator systems from sets of linear relations. We equip $\bb C[X]$ with the usual inner product structure $\ip{f}{g} := \sum_{x\in X} f(x)\bar g(x)$. Consider a set of elements $r_1,\dotsc,r_k\in \bb R[X]$ so that $\ip{r_i}{1_X} = \sum_x r_i(x) =0$ for all $i=1,\dotsc,k$. Let $R = \spn\{r_1,\dotsc,r_k\}$. We define an operator system structure on the quotient space $\bb C[X]/R$ by defining the completely positive maps from $\bb C[X]/R$ to $M_n$ to be exactly the set of linear maps induced by functions $f: X\to M_n^+$ so that $\sum_{x\in X} r_i(x)f(x) = 0$ for all $i=1,\dotsc,k$. Since $1_X\in R^\perp$ we have that the constant function $f(x)\equiv \frac{1}{|X|} I_n$ is a unital, completely positive map, so that $\ucp(\bb C[X]/R, M_n)$ is nonempty for each $n$.
    
    This is a special case of the quotienting construction detailed above, as it is easy to check that $R\subset \bb C[X]$ is a kernel. This construction seems rather simple, but this belies a very rich structure. For example, let us take $|X|= 2n$ and label the elements of $X$ as $s_1,\dotsc, s_n,t_1,\dotsc,t_n$. Consider the single element $r$ defined by $r(s_i)=1$ and $r(t_j)=-1$ for all $i,j=1,\dotsc,n$. It is a result of Kavruk \cite{Kavruk2015} that in this case $\bb C[X]/R$ is canonically completely order isomorphic to the operator subsystem of $C^*(\bb Z/n\bb Z\ast \bb Z/n\bb Z)$ spanned by the generating set $(\bb Z/n\bb Z\ast 1)\cup (1\ast\bb Z/n\bb Z)$.
\end{example}

\begin{example}
    Let $P\subset \bb R^k$ be a closed cone with archimedean order unit $e$.  We can define an operator system structure on $V  = (\bb C^k, P, e)$ by declaring $\ucp(V:M_n)$ to be the set of all linear maps $f: \bb C^k \to M_n$ so that $f(\bb R^k)\subset M_n^h$, $f(e)=I_n$, and $f(p)\in M_n^+$ for all $p\in P$. This is the \emph{maximal operator system structure} $\MAX(V)$ over $V$ as defined in \cite{ptt2011}: the following is proved therein.
\end{example}
    
\begin{prop}
    The operator system $\MAX(V)$ is characterized by the universal property that $\vp: \MAX(V)\to \cc B(H)$ is completely positive if and only if $\vp(P)\subset \cc B(H)^+$.
\end{prop}
    
We record this property by introducing new terminology.
    
\begin{defn} \label{defn:maximal}
    We say that an operator system $E$ is \emph{$k$-maximal} if every $k$-positive map $\vp: E\to \cc B(H)$ is completely positive. For ease of reference, we will say that $E$ is \emph{maximal} if it is $1$-maximal.
\end{defn}
    
\begin{remark} \label{rmk:cone}
    If $K\subset \bb R^\ell$ is a compact, convex set with $0$ as an interior point, we can form the cone $P_K\subset \bb R^{\ell+1}$ to consist of all vectors of the form $te - tv$ where $e = (1,0,\dotsc, 0)$, $v\in 0\oplus K$, and $t\geq 0$. It can be checked that $e$ is an archimedean order unit for $P_K$ and $P_K- P_K = \bb R^{\ell+1}$. 
    
    \begin{prop}
        For the maximal operator system structure on $V_K := (\bb C^{\ell+1}, P_K, e)$, we have that $\ucp(V_K,M_n)$ is identified with the set of all $\ell$-tuples $X_1,\dotsc,X_\ell$ of hermitian elements in $M_n$ satisfying the \emph{linear matrix inequalities} $I_n\succeq v_1X_1 + \dotsc+ v_\ell X_\ell$ for all $v\in K$.
    \end{prop} 
    
    We leave the proof as an exercise to the reader. Conversely, for a closed cone $P$ with archimedean order unit $e$ belonging to the interior, consider the ``slice'' $K = \{p\in P: \ip{p}{e} =1\}$. It can be shown that $P_K$ is affinely isomorphic to $P$, so that all maximal operator systems essentially arise in this way. This is essentially a special case of the Webster--Winkler Duality Theorem \cite{webster-winkler}.
\end{remark}

\begin{example}
    Again, let $P\subset \bb R^k$ be a closed cone with archimedean order unit $e$.  We can define an operator system structure on $V  = (\bb C^n, P, e)$ by embedding $V$ via the evaluation map into the algebra $C(\cc S(V))$ of continuous functions on the set $\cc S(V)$ of all states, which is a closed and bounded, hence compact, subset of $\bb C^n$. This is the \emph{minimal operator system structure} $\MIN(V)$ defined on $V$ as described in \cite{ptt2011}. The following is proved therein.
    
    \begin{prop}
        The minimal operator system structure on $V$ is characterized by the universal property that for every operator system $E$ every positive map $\vp: E\to V$ extends to a completely positive map $\vp: E\to \MIN(V)$.
    \end{prop}
\end{example}

\section{Model Theory of Operator Systems}

A linear matrix $\ast$-polynomial in $n$-variables $x =(x_1,\dotsc,x_n)$ is an expression of the form
\begin{equation*}
   p(x) = a_1\otimes x_1 + \dotsb + a_n\otimes x_n + b_1\otimes x_1^* + \dotsb + b_n\otimes x_n^* + c\otimes 1
\end{equation*}
where $(a_1,\dotsc,a_n,b_1,\dotsc,b_n,c)$ are elements in $M_k$ for some $k$. Equivalently, we may think of $p(x)$ as a $k\times k$ matrix where each entry is a linear $\ast$-polynomial in $1,x_1,\dotsc,x_n$.

\begin{defn}
    We say that a linear matrix $\ast$-polynomial is \emph{hermitian} if $p(x^*)^* = p(x)$ which is to say that $c$ is hermitian and $a_i^* = b_i$ for all $i=1,\dotsc,n$. A linear matrix $\ast$-polynomial is said to be \emph{homogeneous} if $p(0) = 0$, that is, if $c=0$. 
    We say that $p(x)$ has \emph{degree} $k$ if $a_1,\dotsc,a_n,b_1,\dotsc,b_n,c\in M_k$, and we write $\deg(p) = k$.
\end{defn}

We define a \emph{linear matrix inequality} to be an expression of the form $p(x)\succeq 0$ where $p(x)$ is a hermitian linear matrix $\ast$-polynomial. Note that given a system $p_1(x)\succeq 0,\dotsc, p_m(x)\succeq 0$ of linear matrix inequalities we may without loss of generality combine them into a single linear linear matrix inequality via a block diagonal embedding of the matrix coefficients of each $p_i$ into a larger matrix. Many problems in functional analysis can be essentially phrased as the following general problem:

\begin{problem}[Matrix Completion Problem]
    Does there exist $X = (X_1,\dotsc, X_n)\in \cc D^n$ for some domain $\cc D\subset \cc B(H)$ so that $p(X)\succeq 0$? (Here $1$ is interpreted as the unit in $\cc B(H)$.)
\end{problem}

From the perspective of logic, the Matrix Completion Problem can be seen as whether some existential sentence is true over some (hopefully) suitable domain of quantification $\cc D$. Notice that by Proposition \ref{prop:order-norm}, this framework captures problems involving norm estimates as well: if we are interested in the statement
\[\exists{x\in \cc D^n} : \|p(x)\|\leq 1,\] then setting $p^*(x) := p(x^*)^*$ we may write this equivalently as
\[\exists{x\in \cc D^n} : \begin{bmatrix} 1 & p(x)\\ p^*(x) & 1\end{bmatrix}\succeq 0.\]
If $H$ is a finite-dimensional Hilbert space, the Matrix Completion Problem frequently occurs in the context of \emph{semidefinite programming} in terms of whether a given semidefinite program is feasible. (We will not go any further into this here, but we refer the reader to \cite{GartnerMatousek2012, Lovasz2003} for an introduction to modern developments in the theory of semidefinite programming.) 

\subsection{Building the Language} 

One could say that the broad goal of building a model theory for matrix operator systems is to devise a general framework where problems such as the Matrix Completion Problem can be systematically studied.
Following \cite[Appendix B]{gs-kirchberg} we give a description of how to axiomatize operator systems in the context of first order continuous logic for metric structures. (We refer the reader Hart's article in this volume or to \cite[Section 2.1]{model-c-star} for the basics of first order continuous logic for metric structures.) In terms of building the language, this immediately brings to attention the following considerations.

\begin{enumerate}
    \item We will need a collection of sorts $E_1, E_2, E_3, \dotsc$ intended to capture the operator system $E = E_1$ as well as its matrix amplifications so that $E_n$ should be interpreted as $M_n(E)$ when our work is done. Each of these sorts will need domains of quantification $\cc D_r(E_n)$ corresponding to the norm $r$-balls about the origin. Since we are about to give each $E_n$ a real vector space structure, it will suffice to consider $\cc D_1(E_n)$ alone.  
    
    \item We will need another collection of sorts $C_1,C_2,C_3,\dotsc$ for which each $C_n$ will need to be interpreted as the cone of positive elements $E_n^+$ in $E_n$ with domains of quantification to be interpreted as the restriction of the $r$-balls in $E_n$ to $C_n$.
    
    \item Sorts $M_{n,k}$ for the complex $n\times k$ matrices with each domains of quantification the operator norm $r$-balls.
    
    \item Sorts for $\bb C$, $\bb R$, and the non-negative reals $\bb R_{\geq 0}$ with the domains of quantification being the standard $r$-balls.
\end{enumerate}

For each of these sorts we will need the corresponding relational symbols.

\begin{enumerate}
    \item We will need constant symbols $0,1$ in $\cc D_1(E_1)$ for $0$ and the order unit. We require function symbols
    \[\cdot_n: \bb C\times E_n\to E_n,\ +_n: E_n\times E_n\to E_n,\ \textup{and}\ \ast_n: E_n\to E_n\]
    to be used for scalar multiplication, addition, and involution, respectively. Additionally one further set of function symbols
    \[f_{n,k}: M_{n,k} \times E_n \to E_k\] is needed for $(a,x)\mapsto a^*xa$ and another
    \[p_n^{ij}: E_n\to E_1\] for projections giving the matrix coordinates. There are predicate symbols
    \[\|\cdot\|_n: E_n\to \bb R_{\geq 0}\] which we will want to interpret as the norms induced from the matrix ordering as in Proposition \ref{prop:order-norm}.
    
    \item We define constant symbols $0',1'$ in $\cc D_1(C_1)$ and function symbols
    \[\cdot_n': \bb R_{\geq 0}\times C_n\to C_n,\ +_n': C_n\times C_n\to C_n,\ \textup{and}\ g_{n,k}: M_{n,k}\times C_n\to C_k\]
    to be used for scalar multiplication, addition, and $(a,x)\mapsto a^*xa$, respectively.
    
    \item Similarly to the first two items, there will need to be a constant symbol $0$ and function symbols for scalar multiplication, addition, and taking adjoints $\ast: M_{n,k}\to M_{k,n}$ for each $M_{n,k}$.
    
    \item Again, we must define symbols for $0,1$ in $\bb R_{\geq 0}$, $\bb R$, and $\bb C$ with an additional symbol $i$ in $\bb C$, all in the unit domain of quantification. Function symbols are needed for addition, scalar multiplication, and, in the case of $\bb C$, conjugation.
    
    \item Besides these we need a couple of other sets of function symbols
    \[i_n: C_n\to E_n,\] obviously for the inclusion of the positive cone in each $E_n$ and
    \[h_n: E_n\to E_{2n}\] which we want to interpret as 
    \[h_n: x\mapsto \begin{bmatrix} 1_n & x\\ x^* & 1_n\end{bmatrix}\]
    to be used to guarantee our predicate symbols are interpreted as intended.
\end{enumerate}

Pertaining to the use of the $h_n$'s we require the following result.

\begin{lem} \label{lem:dist-to-positives}
    For an operator system $E$ and $x\in E_n$ we have that 
    \[ \|x\|_n \dminus 1 = \dist\left(\begin{bmatrix} 1 & x\\ x^* & 1\end{bmatrix}, E_{2n}^+\right).\]
    Here $x\dminus y := \max\{x-y,0\}$.
\end{lem}

\begin{proof}
    We may assume without loss of generality that $n=1$ and $r := \|x\|>1$. By Proposition \ref{prop:order-norm} we have that \[\begin{bmatrix} 1 & \frac{1}{r}x\\ \frac{1}{r}x^* & 1\end{bmatrix} \in E_2^+;\] hence,
    \[ h(x) := \dist\left(\begin{bmatrix} 1 & x\\ x^* & 1\end{bmatrix}, E_2^+\right) = \dist\left(\begin{bmatrix} 0 & (1-1/r)x\\ (1-1/r)x^* & 0\end{bmatrix}, E_2^+\right).\]
    
    Noting that 
    \[\left\|\begin{bmatrix} a & b\\ b^* & c\end{bmatrix}\right\| = \left\|\begin{bmatrix} 1 & 0\\ 0 & -1\end{bmatrix}\begin{bmatrix} a & b\\ b^* & c\end{bmatrix}\begin{bmatrix} 1 & 0\\ 0 & -1\end{bmatrix}\right\| = \left\|\begin{bmatrix} a & -b\\ -b^* & c\end{bmatrix}\right\|,\]
    we see by simple arithmetic that
    \[\left\|\begin{bmatrix} 0 & y\\ y^* & 0\end{bmatrix} - \begin{bmatrix} a & b\\ b^* & c\end{bmatrix}\right\| = \left\|\begin{bmatrix} 0 & y\\ y^* & 0\end{bmatrix} - \begin{bmatrix} a & -b\\ -b^* & c\end{bmatrix}\right\|.\]
    Hence, by averaging the two expressions and using the triangle inequality we see that $h(x)$ must be approached by elements in $E_2^+$ of the form $\begin{bmatrix} a & 0\\0 & c\end{bmatrix}$. However, by the same argument as before
    \[\left\|\begin{bmatrix} 0 & y\\ y^* & 0\end{bmatrix} - \begin{bmatrix} a & 0\\0 & c\end{bmatrix}\right\| = \left\|\begin{bmatrix} 0 & -y\\ -y^* & 0\end{bmatrix} - \begin{bmatrix} a & 0\\ 0 & c\end{bmatrix}\right\|\]
    which shows by averaging that $a=c=0$ is optimal. Therefore
    
    \[h(x) = \left\|\begin{bmatrix} 0 & (1-1/r)x\\ (1-1/r)x^* & 0\end{bmatrix}\right\|_2= \left(1 - \frac{1}{\|x\|}\right)\|x\|\]
    since
    \[\begin{bmatrix} x & 0\\ 0 & x^*\end{bmatrix} = \begin{bmatrix} 0 & x\\ x^* & 0\end{bmatrix}\cdot\begin{bmatrix} 0 & 1\\ 1 & 0\end{bmatrix}.\qedhere\] 
\end{proof}

We require one further technical result: see \cite[Lemma B.1]{gs-kirchberg}.

\begin{exercise} \label{ex:cone}
    For an operator system $E$ and $x\in E_n^+$ we have that $\dist(-x, E_n^+) = \|x\|_n$.
\end{exercise}

To complete the axiomatization we need axioms that tell us that the symbols enumerated above are interpreted as they should be. These would include, for instance, axioms stating that each $E_n$ has the structure of a $\ast$-vector space, $E_n$ is isomorphic to $M_n(E_1)$, and the $\ast$-vector space structure on $E_n$ is the natural one induced from the $\ast$-vector space structure on $E_1$ via matrix amplification. In addition to these, we need the following, more specialized, axioms.

\begin{enumerate}
    \item There is an axiom that each $\|\cdot\|_n$ is seminorm on $E_n$, and that this predicate \emph{is} the metric predicate, that is,
    \[\sup_{x,y\in \cc D_1(E_n)} |\dist(x,y) - \|x -y\|_n| = 0.\]
    This implies that each $\|\cdot\|_n$ is, in fact, a norm.
    \item We need an axiom stating that $i_n: C_n\to E_n$ is an isometric inclusion, an axiom for ensuring that the range of $i_n$ lies in the hermitian part of $E_n$,
    \[\sup_{x\in \cc D_1(C_n)} \dist(i_n(x),i_n(x)^*)=0,\]
    and an axiom (by way of Exercise \ref{ex:cone}) ensuring that $-i_n(C_n)\cap i_n(C_n) = \{0\}$, 
    \[\sup_{x,y\in \cc D_1(C_n)} \left|\|x\|_n \dminus \dist(i_n(x),-i_n(y))\right| = 0.\]

    \item Finally, there is an axiom stating the content of Lemma \ref{lem:dist-to-positives}, that is, 
    \[\sup_{x\in \cc D_r(E_n)} \left| (\|x\|_n\dminus 1) - \dist(h_n(x),C_{2n})\right|=0.\]
\end{enumerate}

\begin{remark}
    The last axiom is important as it guarantees that $1$ is interpreted as a (complete) order unit. To see this, let $x\in E_n$ with $x = x^*$ and set $r := \|x\|_n$. We have that $\|\frac{1}{r} x\|_n =1$, thus
    \[h_n\left(\frac{1}{r}x\right) = \begin{bmatrix} 1_n & \frac{1}{r} x\\ \frac{1}{r}x & 1_n \end{bmatrix}\in C_{2n}.\] We then have that $v^*h_n(x)v = 1_n + \frac{1}{r}x\in C_n$ where $v = [\frac{1}{\sqrt 2}I_n, \frac{1}{\sqrt 2}I_n]^t$.
\end{remark}

\begin{remark}
    As a consequence of the previous remark we have that $i_n(C_n)$ spans $E_n$ for all $n$. Indeed, for all $x\in E_n$ with $\|x\|_n\leq 1$ we have that 
    \[x = \left((1_n + (x+x^*)/2) - 1_n\right) + \sqrt{-1}\left((1_n - \sqrt{-1}(x-x^*)/2) - 1_n\right)\]
    which expresses $x$ as a linear combination of four elements in the positive cone, each with norm at most $2$.
\end{remark}

So far the axiomatization says that any model can be interpreted as a matrix ordered $\ast$-vector space with order unit (see Remark \ref{rem:matrix-ordered-ast-vector-space}), the only difference from being an abstract operator system then is the issue of whether $1$ is an \emph{archimedean} order unit. To bridge this last gap, we will rely crucially on the fact that the metric structure gives rise to a norm $\|\cdot\|_n$ on each $E_n$.

\begin{lem} \label{lem:archimedean-norm}
    Let $E$ be an matrix ordered $\ast$-vector space with order unit $1$. Setting \[\|x\|_n' := \inf\left\{t>0 : \begin{bmatrix} t1_n & x\\ x^* & t1_n\end{bmatrix}\in E_{2n}^+\right\},\] we have that that $E$ is an abstract operator system if and only if one of the following equivalent conditions holds:
    \begin{enumerate}
        \item $1$ is an archimedean order unit for $E$;
        
        \item $1_n$ is an archimedean order unit for $E_n$ for all $n$;
        
        \item $\|\cdot\|_n'$ is a norm for all $n$.
        
        \item $\|\cdot\|_1'$ is norm.
    \end{enumerate}
\end{lem}

\begin{proof}
    ($1 \Leftrightarrow 2$) Suppose that $1_n$ is not archimedean for some $n$, that is, there exists a non-zero $x\in E_n^h$ so that $r1_n \pm x\in E_n^+$ for all $r>0$. We must have that $y = v^*xv\not= 0$ for some unit vector $v\in \bb R^n$. Seeing that $r1 \pm y = v^*(r1_n \pm x)v\in E_1^+$, we verify that $1$ is not archimedean. The converse is trivial.
    
    ($2\Rightarrow 3$) We leave it as an exercise to check that $\|\cdot\|_n'$ is always a seminorm, that is, $\|x\|_n' = \|-x\|_n'$, $\|\la x\|_n' = |\la|\,\|x\|_n'$, and $\|x+y\|_n'\leq \|x\|_n' + \|y\|_n'$ for all $\la\in \bb C$ and $x,y\in E_n$. For a nonzero $x\in E_n$, we have that \[y := \begin{bmatrix} 0 & x\\ x^* & 0\end{bmatrix}\in E_{2n}^h;\] thus, there is a $t>0$ so that for all $s<t$ it is not the case that $s1_{2n} \pm y\in E_{2n}^+$. It follows that $\|x\|_n' = t$.
    
    ($3\Rightarrow 2$) We can assume without loss of generality that $1$ is not archimedean as witnessed by $x\in E^h$. Now since $x+r1\in E^+$,
    \[ \begin{bmatrix} x+r1 & x+r1\\ x+r1 & x+r1\end{bmatrix} = [1,1]^t (x+r1)[1,1]\in E_2^+\]
    and
    \[\begin{bmatrix} (2r)1 & x+r1\\ x+r1 & (2r)1\end{bmatrix} - \begin{bmatrix} x+r1 & x+r1\\ x+r1 & x+r1\end{bmatrix} = \begin{bmatrix} r1 - x & 0\\ 0 & r1 - x\end{bmatrix}\in E_2^+.\]
    Since $|\|x+r1\|_1' - \|x\|_1'|\leq \|r1\|_1'\leq r$, we conclude that $\|x\|_1' =0$, so $\|\cdot\|_1'$ is not a norm.
    
    Finally, ($4\Rightarrow 3$) is trivial and ($1\Rightarrow 4)$ is a special case of ($2\Rightarrow 3$). \qedhere

\end{proof}

With this lemma now in hand, we see that the axiomatization of abstract operator systems as metric structures is complete. We note in passing that by the Choi--Effros representation theorem for abstract operator systems each $E_n$ is metrically complete in the $\|\cdot\|_n'$-norm, so there was nothing lost in insisting upon this in the metric structure language, though nothing is gained by doing so.

\subsection{Formulas, Theories, and Models}

Now that the language is built we briefly discuss the basic aspects of the continuous model theory of operator systems. We refer the reader to \cite[Chapter 2]{model-c-star} and \cite{model-metric} for a more in-depth treatment.

Let $\cc L$ be a language for metric structures as described in \cite[Section 2.1]{model-c-star}. To each domain of quantification in each sort we assign an infinite number of variables. 

\begin{defn} {}\
    
    \begin{enumerate}
        \item An \emph{atomic formula} $\vp(x_1,\dotsc,x_n)$ is an expression in the language, built using finitely many variables and function symbols, which terminates with the application of a predicate. (We will assume all predicates take values in $\bb R^+$.) Since each variable $x_i$ has an assigned domain of quantification $\cc D_i$, we have that $f$ has domain
        \[\domain(\vp) := \cc D_1\times \dotsb \times \cc D_n.\]
        Since each variable comes with a (bounded) domain of quantification and each function symbol has a modulus of uniform continuity we can assign a bounded \emph{range} $[0,K]\subset \bb R^+$ to $\vp$.
        
        \item A \emph{connective} is a uniformly continuous function $f: (\bb R^+)^k\to \bb R^+$.
        
        \item If $\vp_1,\dotsc,\vp_k$ are atomic formulas over a common set of variables $x = (x_1,\dotsc,x_n)$, we can define a \emph{quantifier-free formula} as an expression of the form
        $g(x) = f(\vp_1(x),\dotsc,\vp_k(x))$ where $f$ is a connective.
        
        \item A \emph{formula} $h(x)$ is an expression of the form 
        \[h(x) = Q_\ell \dotsb Q_1\, g(x_1,\dotsc,x_n)\]
        where $\ell\leq n$, $g$ is a quantifier-free formula, and $Q_i$ is either the \emph{existential quantifier} $\inf_{x_i\in \cc D_i}$ or the \emph{universal quantifier} $\sup_{x_i\in \cc D_i}$. A formula is a \emph{sentence} if all variables are quantified, that is, if $\ell =n$. Similarly to atomic formulas, each formula has an associated domain and bounded range.
        
        \item To each formula $h(x) = h(x_1,\dotsc,x_n)$ and each metric $\cc L$-structure $M$, we can associate an \emph{interpretation} $h(x)^M$ of $h$ which is the function from $\cc D_1(M)\times \dotsc \times D_n(M)$ to $\bb R^+$ determined by interpreting all variables, function symbols, domains of quantification, etc., in the language $\cc L$ in $M$. If $h$ has range $[0,K]$, then this means that $h(x)^M\in [0,K]$ for all $\cc L$-structures $M$ and all $x_1,\dotsc,x_n\in M$.
    \end{enumerate}
\end{defn}

\begin{example}
    Let $\cc L$ be the language of operator systems described above. As we will explain below in Remark \ref{rmk:no-positives} the sorts corresponding to the positive cones do not add any expressive power in terms of defining formulas, so it suffices to build formulas only using the sorts $(E_n)$. Since the matrix entries of any variable in $E_n$ are interpreted as being in the sort $E=E_1$, it suffices to only consider formulas with variables in this sort alone. Thus, in the language of operator systems, atomic formulas are effectively expressions of the form $\|p(x)\|_d$ where $p(x)$ is a linear matrix $\ast$-polynomial of degree $d$, where each variable is restricted to a domain of quantification in $E_1$. It is a bit annoying to have to deal with the domains of quantification at this level, rather than just assigning each variable to the (unbounded) sort $E_1$. We will take this view, so each $\|p(x)\|_d$ is technically what we will term an \emph{unbounded atomic formula}. 
\end{example}

\begin{defn} \label{def:model}
    Let $\cc L$ be a language for metric structures.
    \begin{enumerate}
        \item Any collection $T$ of $\cc L$-sentences is called a \emph{theory}.
        
        \item An $\cc L$-structure $M$ \emph{models} $T$, written $M\models T$, if $h^M = 0$ for all sentences $h\in T$.
        
        \item A theory $T$ is said to be \emph{consistent} if there is an $\cc L$-structure $M$ so that $M\models T$.
        
        \item The \emph{theory} of an $\cc L$-structure $M$, denoted $\Theory(M)$ is the collection of all $\cc L$-sentences $h$ so that $h^M = 0$. Cleary, $M\models \Theory(M)$.
        
        \item We write $\Model(T)$ for the class of all models of a theory. We say that a class $\cc C$ of $\cc L$-structures is \emph{elementary} if $\cc C = \Model(T)$ for some theory $T$.
    \end{enumerate} 
\end{defn}

We will discuss more about elementary classes at the end of the next section.

\subsection{Ultraproducts}

Let $I$ be an arbitrary set, and let $(E_i)_{i\in I}$ be a collection of (concrete) operator systems indexed by $I$. We define the \emph{direct product} $\prod_{i\in I} E_i$ to be the set of all bounded functions $x: I\to \bigsqcup_{i\in I} E_i$, the disjoint union, with $x_i\in E_i$ for all $i\in I$. This is a Banach space $\ast$-vector space under pointwise addition, scalar multiplication, and involution and norm $\|x\| = \sup_{i\in I} \|x_i\|_{E_i}$. For representations $E_i\subset B(H_i)$, we see that $\prod_{i\in I} E_i$ is a concrete operator system isometrically represented on the Hilbert space direct sum $\bigoplus_{i\in I} H_i$ via 
\[x(\oplus_{i\in I} \xi_i) := \oplus_{i\in I} x_i\xi_i.\]
It is left as an exercise to check that under this concrete representation we have that $(\prod_{i\in I} E_i)^+ = \prod_{i\in I} E_i^+$ with unit $1=(1_i)_{i\in I}$ where $1_i\in E_i$ is the unit.

For the rest of the section $I$ will be a fixed directed set and $\cc U$ an ultrafilter on $I$. Our task will be to give a ``concrete'' definition of the ultraproduct of the operator systems $(E_i)_{i\in I}$ and then explain how this ultraproduct is the ultraproduct at the level of the language of operator systems that we have developed. Consider an arbitrary collection $\phi = (\phi_i)_{i\in I}$ of matrix states $\phi_i\in \cc S_n(E_i)$. Each such $\phi$ induces a matrix state \[\phi_{\cc U}: \prod_{i\in I} E_i\to M_n,\ \textup{defined by}\ \phi_\cc U(x) := \lim_{\cc U} \phi_i(x_i).\]

\begin{exercise} \label{ex:ultra-kernel}
    Show that for $\phi= (\phi_i)\in \prod_{i\in I} \cc S(E_i)$
    \[\bigcap_{\phi = (\phi_i)} \ker(\phi_{\cc U}) = \left\{(x_i)\in \prod_{i\in I} E_i : \lim_{\cc U}\|x_i\| = 0\right\} =: \cc J.\] Hence, $\cc J$ is a kernel. 
\end{exercise}

\begin{exercise} \label{ex:ultra-archimedean}
    Show that the class of $1$ is an archimedean order unit for the ordered $\ast$-vector space quotient; hence, the evaluation map $\hat\cdot: \prod_{i\in I} E_i/\cc J\to (\prod_{i\in I} E_i/\cc J)^{**}$ is faithful and $\prod_{i\in I} E_i/\cc J$ itself is equipped with a quotient operator system structure.
\end{exercise}

\begin{defn} \label{defn:ultraproduct}
    We define the \emph{ultraproduct} $\prod_{\cc U} E_i$ of the operator systems $(E_i)_{i\in I}$  to be the operator system quotient $\prod_{i\in I} E_i/\cc J$. In other words, $\prod_{\cc U} E_i$ is the operator system structure defined by the matrix states
    \[\ucp\left(\prod_{\cc U} E_i, M_n\right) = \left\{\phi_{\cc U} : \phi \in \prod_{i\in I} \ucp(E_i,M_n)\right\}.\]
\end{defn}

To check that this ultraproduct is an ultraproduct of metric structures in the language of operator systems by \cite[Chapter 5]{model-metric} it suffices to show that this construction is compatible with the metric ultraproduct for each sort, that is,
\begin{equation*} \label{eq:dissection}
    \cc D_r\left(\prod_{\cc U} M_n(E_i)\right) = \prod_{\cc U} \cc D_r(M_n(E_i)),\quad \cc D_r\left((M_n(\prod_{\cc U} E_i))^+\right) = \prod_{\cc U} \cc D_r(M_n(E_i)^+)
\end{equation*}
for all $r>0$ and $n=1,2,\dotsc$. In fact, checking this for $r=1$ suffices due to the existence of a metric compatible scalar multiplication operation, and we may further assume without loss of generality that $n=1$. For the first equation this follows immediately from Exercises \ref{ex:ultra-kernel} and \ref{ex:ultra-archimedean}. 
For the second equation suppose, by way of contradiction, that there is $x=(x_i)\in \cc D_1\left(\prod_{\cc U} E_i\right)^+$ so that $\lim_{\cc U}\dist(x_i,\cc D_1(E_i^+)) =:\alpha >0$. We can clearly assume that each $x_i$ is hermitian. Since $\|x_i - 1\|\geq \alpha/2$ for $i\in \cc U$ generic, we have that $\|x_i\|\geq 1 + \alpha/2$ for $i\in \cc U$ generic. However, this contradicts that $1=\|x\| = \lim_{\cc U} \|x_i\|$, and we have established the second equation holds.

\begin{notation}
    In the case that each $E_i$ is identified with $E$, we write $E^{\cc U} := \prod_{\cc U} E_i$ and refer to $E^{\cc U}$ as the \emph{$\cc U$-ultrapower} of $E$.
\end{notation}

Let $\cc C$ be a class of $\cc L$-structures. We say that $\cc C$ is closed under taking \emph{ultraroots} if $E\in \cc C$ whenever $E^{\cc U}$ for some ultrafilter $\cc U$. The following appears as \cite[Theorem 2.4.1]{model-c-star}

\begin{prop} \label{prop:semantic-test}
    A class $\cc C$ of $\cc L$-structures is elementary (in the sense of Definition \ref{def:model}) if and only if $\cc C$ is closed under isomorphisms, ultraproducts, and taking ultraroots.
\end{prop}

\begin{example}
    By essentially the same reasoning as the proof of the Choi--Effros representation theorem (Theorem \ref{thm:choi-effros}) we may see that every operator system $E$ admits a complete order embedding into $\prod_{\cc U} M_{n_i}$ for some non-principal ultrafilter on some directed set $I$. Thus, the class of all operator systems is the smallest elementary class containing all matrix algebras which is closed under taking substructures.
\end{example}

\begin{example}
    We say an operator system $E$ is \emph{minimal} if it is a subsystem of an abelian C$^*$-algebra; equivalently, if $\widehat E\subset C(\cc S(E))$ is a complete order embedding. Since the operator system ultraproduct of unital (abelian) C$^*$-algebras is again a unital (abelian) C$^*$-algebra under the natural multiplicative structure, we have that the class of minimal operator systems is elementary.
\end{example}

\begin{problem}
    Find an axiomatization of the class of minimal operator systems.
\end{problem}

By the same token, if $\cc C$ is an elementary class of unital C$^*$-algebras, then the class of all subsystems of elements of $\cc C$ is an elementary class of operator systems.

\begin{question}
    Is there an elementary class of operator systems which is closed under taking subsystems which is not the class of all subsystems of some elementary class of unital C$^*$-algebras?
\end{question}

\subsection{Definability}

From a practical viewpoint, the most important task at hand once a theory $T$ is constructed is to begin to explore the definable sets in the models of $T$. In short, this is because the definable sets are exactly those which it is permissible to quantify over, so having a large class of natural sets being definable allows for a great range of intuitive constructions of formulas, ``breathing life'' into the theory. From the point of view of a working analyst, the concept of definability is perhaps the key feature of continuous model theory as to say that a subset of a structure is definable is to say that it possesses stability/rigidity under small perturbations. In this light many important results in the theory of operator systems and C$^*$-algebras can be seen as establishing the definability of sets of elements satisfying certain formulas. Our treatment of definability here is directly taken from \cite[Chapter 3]{model-c-star}, though we will center our discussion here on the practical definition of definability for sets which is afforded by the Beth Definability Theorem: see \cite[Section 4.2]{model-c-star} or \cite[Theorem 9.32]{model-metric}.

Let $T$ be a theory and $\cc C$ an elementary class of models of $T$. Let $F$ be an assignment to each $A\in \cc C$ a closed subset $F(A)\subset \prod_{i=1}^d \cc D_{r_i}(A)$, where $\cc D_{r_1}(A), \dotsc, \cc D_{r_d}(A)$ are domains in $A$ for $d, r_1,\dotsc, r_d$ fixed. We write 
\[\codomain(F) := \prod_{i=1}^d \cc D_{r_i}\]
so that the interpretation of $\codomain(F)$ over $A\in \cc C$ is
\[\codomain(F)^A = \prod_{i=1}^d \cc D_{r_i}(A).\]

\begin{defn}
    We say that $F$ is a \emph{uniform assignment} if $F$ is a functor, that is, for every homomorphism $\vp:A\to B$ with $A,B\in \cc M$ we have that $\vp(F(A))\subset F(B)$.
\end{defn}

\begin{remark}
    In the language of operator systems defined above we have that the morphisms are always the unital, completely positive maps which are, rather conveniently, all contractions.
\end{remark}

When considering the model theory of operator systems (or C$^*$-algebras) it makes sense to introduce the following unbounded variant of a uniform assignment. Let $\cc C$ be an elementary class in the theory of operator systems. Suppose that $F$ assigns to every $A\in \cc C$ a closed subset 
\[F(A)\subset M_{r_1}(A) \times \dotsb \times M_{r_k}(A)\times M_{s_1}(A)^+\times\dotsb \times M_{s_l}(A)^+\times \bb C^d.\] We say that $F$ is then an \emph{unbounded uniform assignment} if its restriction to every product of domains is a uniform assignment.

\begin{defn}
    Let $\cc C$ be an elementary class of operator systems and $F$ be an unbounded uniform assignment on $\cc C$. We say that $F$ is \emph{definable} if for every ultraproduct $\prod_{\cc U} E_i$ of elements $E_i\in \cc C$ we have that
    \[ F(\prod_{\cc U} E_i) = \prod_{\cc U} F(E_i).\]
\end{defn}

To every (unbounded) formula $f(x)$, there is a natural (unbounded) uniform assignment $Z_f$ which assigns to every $E\in \cc C$ the \emph{zero set}
\[Z_f(E) = \{x\in \domain(f)^E : f(x) = 0\}.\]

\begin{exercise} \label{ex:definable-epsilon-delta}
    Let $f$ be a formula. We have that $Z_f$ is definable over $\cc C$ if and only if for every $\e>0$ there is $\delta>0$ so that for all $A\in \cc C$ and $x\in \domain(f)^A$, we have that $f(x)^A\leq \de$ implies that $\dist(x,Z_f(A))\leq \e$.
\end{exercise}

\begin{thm}[Beth Definability Theorem] \label{thm:beth}

Let $\cc C$ be the class of models of a theory $T$, and let $F$ be a uniform assignment. Writing
\[p_F(x) := \dist(x,F),\quad x\in\codomain(F),\]
we interpret $p_F(x)^A$ as $\dist(x,F(A))$ for $x\in \codomain(F)^A$.
We have that $F$ is definable if and only if there is a sequence of formulas $f_1,f_2,\dotsc$ with $\domain(f_i) = \codomain(F)$ so that
\[\sup_{A\in \cc M} \sup\left\{|p_F(x)^A - f_n(x)^A| : x\in \codomain(F)^A\right\}\leq 1/n\]
for all $n=1,2,3,\dotsc$.

\end{thm}

We refer the reader to \cite[Section 4.2]{model-c-star} or \cite[Chapter 9]{model-metric} for a proof.

\begin{defn}
    If $\cc M$ is the class of models for a theory $T$ and $p$ is a predicate, we say that $p$ is \emph{definable} if there is a sequence of formulas $f_1,f_2,\dotsc$ with $\domain(f_i) = \domain(p)$ so that
    \[\sup_{A\in \cc M} \sup\left\{|p(x)^A - f_n(x)^A| : x\in \domain(p)^A\right\}\leq 1/n.\]
\end{defn}

The content of the Beth Definability theorem is then to say that if $F$ is a definable functor, then 
$p_F(x) = \dist(x,F)$ is a definable predicate. In the unbounded case, this means $p_F$ restricted to any product of domains of quantification in $\codomain(F)$ is a definable predicate.

\begin{example}
    The hermitian elements in any operator system form a definable set.
\end{example}

\begin{example} \label{ex:cone-predicate}
    For each $n$ we have that $d_n^+(x) := \dist(x, i_n(C_n))$, $x\in E_n$, is a definable predicate.
\end{example}

\begin{remark} \label{rmk:no-positives}
    In fact, something even stronger may be said. Consider the (unbounded) quantifier-free formula
    \[p_{t,n}(x) := \|2x - t1_n\|_n \dminus t\]
    for $t\geq 0$ with domain $E_n^h$. We have that $p_{t,n}(x)=0$ if and only if $\|x\|_n\leq t$ and $x\in C_n$. This is because for a hermitian element $x\in E_n$ we have that $\|2x - t 1_n\|_n\leq t$ if and only if $0\preceq x\preceq t 1_n$. Therefore, every formula in the language of operators systems can be replaced with an equivalent formula where all quantifiers have domain in $E_n$ for some $n$. That is, it is never necessary to quantify over the positive cones.
\end{remark}

\begin{prop} \label{prop:lmi-beth}
     For each homogeneous, hermitian linear matrix $\ast$-polynomial $p(x)$ of degree $d$ in $x = (x_1,\dotsc,x_n)$, the set of tuples $X=(X_1,\dotsc,X_n)$ in $E_k$ or $E_k^h$ satisfying $1_d\otimes 1_k \succeq p(X)$ is a definable set.
\end{prop}

\begin{proof}
    For ease of notation we will write $1$ for $1_d\otimes 1_k$. Suppose that $\dist(1- p(X),C_{dk})< \e$. Since $1 - p(X)$ is hermitian, we claim \[(1+\e)1 \succeq p(X).\] Indeed, since there exists $Y\in C_{dk}$ so that $\|1 - p(X) - Y\|_{dk}\leq \e$, we have that 
    \[\e 1\succeq 1-p(X) - Y\succeq -\e 1,\ \textup{so}\ 1 - p(X) \succeq Y - \e 1\succeq -\e 1.\]
    Setting $X' = \frac{1}{1+\e}X$, by linearity and homogeneity $1\succeq p(X')$ and $\|X_i' - X_i\|_k\leq \e\|X_i\|_k$ for $i=1,\dotsc,n$. The result now follows by Exercise \ref{ex:definable-epsilon-delta}. \qedhere
    
\end{proof}

\begin{exercise}
    Let $p(x)$ be a homogeneous, hermitian linear matrix $\ast$-polynomial. If $p(x)\succeq cI$ has a solution over $\bb R$ for some $c>0$, then the set of tuples $X=(X_1,\dotsc,X_n)$ in $E$ or $E^h$ satisfying $p(X)\succeq 0$ is definable.
\end{exercise}

\begin{question}
    If $p(x)$ is a homogeneous linear matrix $\ast$-polynomial, is the set of all tuples $X_1,\dotsc,X_n$ in $E_k$ satisfying $1\succeq p(X)$ definable? If not, is there a natural class of operator systems or C$^*$-algebras where all such sets are definable?
\end{question}

\begin{remark}
    We may define a linear operator $\ast$-polynomial to be an expression of the form $q(x) = A_1\otimes x_1 + \dotsb + A_n\otimes x_n + B_1\otimes x_1^* + \dotsb + B_n\otimes x_n^* + C\otimes 1$ where $A_1,\dotsc,A_n,B_1,\dotsc,B_n,C\in \cc B(H)$. All terminology from linear matrix $\ast$-polynomials carries over to this more general context equally well. Just as an finite family of linear matrix $\ast$-polynomials can be combined into a single linear matrix $\ast$-polynomial, any infinite family of linear matrix $\ast$-polynomials can be combined into a single linear operator $\ast$-polynomial.
    
    The main point is that by the same proof as Proposition \ref{prop:lmi-beth}, we have the following.
    
    \begin{prop} \label{prop:lmi-beth-op}
        For each homogeneous, hermitian linear operator $\ast$-polynomial $q(x)$ in $x = (x_1,\dotsc,x_n)$, the set of tuples $X_1,\dotsc,X_n$ in $E_k$ or $E_k^h$ satisfying $1_H\otimes 1_k \succeq q(X)$ is definable.
    \end{prop}
    
    Via block diagonal embedding we have the following consequence.
    
    \begin{cor}
        Let $\{p_i : i\in I\}$ be a collection of homogeneous, hermitian linear matrix $\ast$-polynomials in $x = (x_1,\dotsc,x_n)$. The set of tuples $X_1,\dotsc,X_n$ in $E_k$ or $E_k^h$ satisfying $1_{d_i}\otimes 1_k \succeq p_i(X)$ for all $i\in I$ is definable.
    \end{cor}
\end{remark}

For an operator system $E$ we may regard the set $\ucp(M_n,E)$ as a closed subset of $\cc D_1(E)^{n^2}$ via the correspondence $\vp \leftrightarrow (\vp(e_{ij}))_{ij}$. The following result was first observed in \cite[Section 5.8]{model-c-star}.

\begin{prop} \label{prop:ucp-mn-e-def}
    For each $n$ we have that $\ucp(M_n,E)$ is a definable set for the class of all operator systems.
\end{prop}

\begin{proof}
    We have by Lemma \ref{lem:adjoint} that $\cp(M_n,E)\cong \cp(\bb C, M_n(E))$. This correspondence identifies $\ucp(M_n,E)$ with the set of all $X\in M_n(E)^+$ with $\tr(X) =\sum_i X_{ii} =1$. Suppose $E = \prod_{\cc U} E_i$. We have by the remarks after Definition \ref{defn:ultraproduct}  that for any $X\in M_n(E)^+$ there is a net $X_i\in M_n(E_i)^+$ with $X = (X_i)$. Given any $\e>0$ we have that $\|\tr(X_i) - 1\|\leq \e$ for $i\in \cc U$ generic. Let $b_i := \tr(X_i)\in E^+$. Since $\|b_i-1\|\leq \e$ we have that $b$ is invertible with $\|b_i^{-1/2}-1\|<\e^{1/2}$. Setting 
    \[[X_i']_{kl} := b_i^{-1/2}[X_i]_{kl}b_i^{-1/2}\] we have that $X_i'\in M_n(E_i)^+$ with $\tr(X_i')=1$ and $\|X_i - X_i'\|\leq 2n^2\e^{1/2}$.
\end{proof}

\begin{defn} \label{defn:positive}
    We say that a formula $f$ is \emph{positive} if it is built from atomic formulas only using connectives $\theta(x_1,\dotsc,x_n)$ satisfying $\theta(x_1,\dotsc,x_n)\leq \theta(y_1,\dotsc,y_n)$ when $x_i\leq y_i$, $i=1,\dotsc, n$.
\end{defn}

\begin{remark}
    Note that the proof above shows that $\ucp(M_n,E)$ is explicitly the zero set of the quantifier-free, positive formula $f(X) := \|\tr(X) -1\|$ where $\domain(f) = \cc D_{n^2}(C_n)$.
\end{remark}

Let $S_n$ be the uniform assignment with codomain $\cc D_1(E)^n\times \bb C^n$ which assigns to each operator system the (closed) subset of all $(x_1,\dotsc,x_n,\la_1,\dotsc,\la_n)$ so that $x_i\mapsto \la_i$ extends to a state on $E$. The following result is essentially contained in \cite[Section 5.8]{model-c-star}, though we give a different proof here.

\begin{prop} \label{prop:state-def}
    We have that $S_n$ is the zero set of a uniform limit of quantifier-free formulas. (We refer to this condition on $S_n$ as being \emph{quantifier-free definable}.)
\end{prop}

\begin{proof}
     Let $\bb D\subset \bb C$ be the closed unit disk. For convenience we write $x = (x_{-n},\dotsc,x_n)$ with $x_0=1$ and $x_{-i} = x_i^*$ and $\la = (\la_{-n},\dotsc, \la_n)$ with $\la_{-i} = \bar\la_i$ and $\la_0=1$. Let 
    \[f_n(x,\la) := \sup_{a_i\in \bb D} \left(\|\sum_i a_i\la_i\| \dminus \|\sum_i a_ix_i\|\right)\]
    and note that $f_n$ is a uniform limit of quantifier-free formulas. (Recall $x\dminus y = \max\{x-y,0\}$.) Indeed, let $(K_p)$ be a finite $1/p$-net in $\bb D$, and note that $f_n$ is a uniform limit of the formulas
    \[f_{n,p}(x,\la) = \max_{a_i\in K_p} \left(\|\sum_i a_i\la_i\| \dminus \|\sum_i a_ix_i\|\right)\]
    by an application of the triangle inequality.
    
    From now on, let us fix an operator system $E$. We claim that $f_n(x,\la)^E=0$ if and only if $x_i\mapsto \la_i$ extends to a state on $E$. By Proposition \ref{prop:contraction-positive}, Lemma \ref{lem:state-cp}, and Theorem \ref{thm:arveson} it suffices to check that $f_n(x,\la)^E = 0$ implies that $\vp:x_i\mapsto\la_i$ defines a unital, $\ast$-linear contraction on $F := \spn\{x_{-n},\dotsc,x_n\}$. Observe that $f_n(x,\la)^E=0$ is equivalent to 
    \[\|\sum_i a_ix_i\|\geq \|\sum_i a_i\la_i\|\ \textup{for all}\ a_{-n},\dotsc,a_n\in \bb D.\]
    Suppose that $\sum_i a_i x_i = \sum_i b_ix_i$ so that $\sum_i \frac{1}{2}(a_i - b_i) x_i = 0$; thus, $\sum_i \frac{1}{2}(a_i - b_i)\la_i =0$. This shows that $\vp: F\to \bb C$ is a well-defined, unital, $\ast$-linear map. Finally, $\{\sum_i a_ix_i : a_i\in \bb D\}$ contains a sufficiently small ball about the origin in $F$, so $\vp$ is a contraction.
\end{proof}

Since $M_n(\bb C)\subset M_n(E)$ canonically $\ast$-isometrically for any operator system $E$, we have the following corollary. 

\begin{cor} \label{cor:ucp-e-mn-def}
     Let $S_{n,k}$ be the uniform assignment with codomain $\cc D_1(E)^n\times \cc D_1(M_k(E))^n$ which assigns to each operator system the (closed) subset of all $(x_1,\dotsc,x_n,y_1,\dotsc,y_n)$ so that $y_1,\dotsc,y_n\in M_k(\bb C)$ and $x_i\mapsto y_i$ extends to a matrix state. We have that $S_{n,k}$ is quantifier-free definable.
\end{cor}
 
\begin{remark} \label{rmk:nuclear-quotient}
    The previous proofs do not show that the $S_{n,k}$ are \emph{positive} quantifier-free definable since $\dminus$ is not a positive connective. In fact, it is not possible for $n\geq 5$ for $S_{n,k}$ to be positive quantifier-free definable for all $k$. This is because such a result would show that any quotient of a nuclear operator system (see section \ref{sec:exact-nuclear} below) would be nuclear, essentially following the reasoning of \cite[Section 5.14]{model-c-star}, but there is a five-dimensional nuclear operator system with a non-nuclear quotient which is a variant of Example \ref{ex:linear-quotient}; see \cite{Kavruk2015}.
\end{remark}

\subsection{Model Theory of C$^*$-Algebras}

We will treat the model theory of C$^*$-algebras in a much more abbreviated fashion than operator systems. We refer the reader to Szabo's article in this volume for background on the basic theory of C$^*$-algebras and to \cite{model-c-star, fhs-ii} for an in-depth treatment of the model theory. In brief, the language of C$^*$-algebras is the language of a normed, complex algebra $A$ with a unary involution function $\ast$ satisfying the axiom
\[\|x^*x\|= \|x\|^2\ \textup{for all}\ x\in A.\]
The sorts are, naturally, axiomatized to be the balls around the origin. This is all spelled out in \cite[Section 3.1]{fhs-ii} or \cite[Section 2.2]{model-c-star}.
The atomic formulas in the theory of C$^*$-algebras are of the form $\|p(x)\|$ where $p$ is a noncommutative $\ast$-polynomial in the variables $x=(x_1,\dotsc,x_n)$.

Let $A$ be a C$^*$-algebra. We say that an element $x\in A$ is \emph{positive} if $x = y^*y$ for some $y\in A$. (Importantly, $y$ can be chosen with $\|y\| = \sqrt{\|x\|}$.) It is well known that the set $A^+$ of positive elements forms a closed cone. In the case that $A$ is a unital C$^*$-algebra then the unit is an archimedean order unit with respect to $A^+$. As in the case of operator systems, this positivity structure is crucial to the theory of C$^*$-algebras.

\begin{exercise}
    We have that $A^+$ is an (unbounded) definable set in the language of C$^*$-algebras; thus, $\dist(x,A^+)$ is a definable predicate.
\end{exercise}

As noted above, every unital C$^*$-algebra is an operator system under the natural positivity structure imposed on $M_n(A)$ by the fact that $M_n(A)$ is again a C$^*$-algebra. We would now like to establish that every property about $A$ that can be expressed in the language of operator systems can be expressed in the language of C$^*$-algebras. This reduces to proving the following result, due to Lupini, which can be found in \cite[Appendix C]{gs-kirchberg}.

\begin{prop} \label{prop:os-definable}
    Let $\cc C$ be the elementary class of unital C$^*$-algebras. The uniform assignment which sends $A\in \cc C$ to the norm unit ball of $M_n(A)$ considered as a subset of $\cc D_1(A)^{n^2}$ is definable.
\end{prop}

\noindent Since by Remark \ref{rmk:no-positives} every formula in the language of operator systems is equivalent to one quantified only over balls in $E_n$ for some $n$, it follows that the operator system structure on $A$ is definably interpretable in the language of C$^*$-algebras.

We begin by noting the following lemma.

\begin{lem} \label{lem:sos-definable}
    The set of all tuples $(x_1,\dotsc,x_n)$ so that $\|\sum_{i=1}^n x_i^*x_i\|\leq 1$ is definable.
\end{lem}

\begin{proof}
    Since $\sum_{i=1}^n x_i^*x_i$ is positive, we have that $\|\sum_{i=1}^n x_i^*x_i\|\leq 1$ if and only if \[\sum_{i=1}^n x_i^*x_i\preceq 1.\]
    The proof now follows along the similar lines to the proof of Proposition \ref{prop:lmi-beth}.
\end{proof}

\begin{proof}[Proof of Proposition \ref{prop:os-definable}] 

For tuples $a,b\in A^n$ define $\ip{a}{b} := \sum_{i=1}^n a_i^*b_i\in A$ and $\|a\| := \|\ip{a}{a}\|^{1/2}$. This gives $A^n$ the structure of a left Hilbert C$^*$-module in the sense of \cite{Lance1995}. By \cite[Proposition 1.1]{Lance1995}, the following variant of the Cauchy--Schwarz inequality is valid:
\[ \|\ip{a}{b}\|\leq \|a\|\,\|b\|.\]
Note that there is a natural action of $M_n(A)$ on $A^n$ by left multiplication and that $\ip{xa}{b} = \ip{a}{x^*b}$ where $x^*$ is the usual adjoint of $x$ in $M_n(A)$.

For $x\in M_n(A)$ we define the norm
\[ \|x\|_A := \sup\left\{\|\ip{xa}{b}\| : \|a\|,\ \|b\|\leq 1\right\} = \sup\{\|xa\| : \|a\|\leq 1\},\]
where the last equality is a consequence of the Cauchy--Schwarz inequality above. It is therefore the case that 
\begin{equation*}
    \begin{split}
        \|x^*x\|_A &= \sup\left\{\|\ip{x^*xa}{b}\| : \|a\|,\ \|b\|\leq
        1\right\}\\
        &= \sup\left\{\|\ip{xa}{xb}\| : \|a\|,\ \|b\|\leq 1\right\}\\
        &\geq \sup\left\{\|xa\|^2 : \|a\|\leq 1\right\} = \|x\|_A^2.
    \end{split}
\end{equation*}
It is straightforward to check that $\|xy\|_A\leq \|x\|_A\|y\|_A$ for all $x,y\in M_n(A)$, hence $\|x^*x\|_A = \|x\|_A^2$. That $M_n(A)$ is complete in the $\|\cdot\|_A$-norm is just an application of the triangle inequality and completeness of $A$. Thus, we see that $\|\cdot\|_A$ is a C$^*$-norm on $M_n(A)$. It must then be the case that $\|\cdot\|_A$ is identical to the operator norm on $M_n(A)$ as C$^*$-norms are unique as a consequence of the fact that every $\ast$-homomorphism of C$^*$-algebras is contractive. Since $\|\cdot\|_A$ is definable by Lemma \ref{lem:sos-definable}, the result now follows. \qedhere

\end{proof}

\begin{remark} Having just established that every formula in the language of operator systems is definably interpretable in the language of C$^*$-algebras, it is natural to ask whether the converse is true. The following result shown in \cite{gs-axiom} is a step in this direction.
\begin{prop}
    The elementary class $\cc C$ of all unital $C^*$-algebras is an elementary class in the language of operator systems.
\end{prop}

\begin{question} \label{question:mult-os-definable}
    Is the multiplication operation on C$^*$-algebras definable in the language of operator systems?
\end{question}

By Proposition \ref{prop:semantic-test}, establishing that $\cc C$ forms an elementary class is equivalent to establishing that if $E$ is an operator system so that the operator system structure on some ultrapower $E^{\cc U}$ is a reduct of some (unital) C$^*$-algebra structure on $E^{\cc U}$, then the same is true for $E$. We note that while a positive solution to Question \ref{question:mult-os-definable} would show that the languages of operator systems and unital C$^*$-algebras are bi-interpretable, the quantifier complexity of equivalent formulas may differ depending on the language chosen.
\end{remark}

\section{Model Theory and Finite-Dimensional Approximation Properties} 

 The importance of finite-rank approximations to the theory of Banach spaces and its connection with the theory of Banach-space tensor products was first made apparent by the seminal work of Grothendieck \cite{grothendieck}. In the case of C$^*$-algebras, the theory of finite-dimensional (matrix) approximations and the largely parallel theory of tensor products of C$^*$-algebras were developed in foundational works of Arveson, Choi, Connes, Effros, Kirchberg, Lance, Takesaki, and Tomiyama, among others, and these ideas and concepts continue to form a central part in the theory of operator algebras and operator systems. In this last section, we survey some of the connections that have been made between the model theory of operator systems and their finite-dimensional approximation properties. We refer the reader to the books of Brown and Ozawa \cite{BrownOzawa} and Pisier \cite{pisier-ck} for an in-depth treatment of the connections between finite-dimensional approximation properties and tensor products of C$^*$-algebras and operator systems.

\subsection{Nuclear and Exact Operator Systems} \label{sec:exact-nuclear}

\begin{defn}
    An operator system $E$ is said to be \emph{nuclear} if there are nets $\vp_i : E\to M_{n_i}$ and $\psi_i: M_{n_i}\to E$ of unital, completely positive maps so that $\psi\circ\vp$ converges in the the pointwise-norm topology to the identity, that is, $\lim_i\|\psi_i\circ\vp_i(x) -x \|=0$ for all $x\in E$.
\end{defn}

\begin{remark}
    Technically nuclearity of a  C$^*$-algebra $A$ refers to the condition that for every C$^*$-algebra $B$ there is exactly one cross norm on the algebraic tensor $A\odot B$ which can be completed to a C$^*$-algebra: see Szabo's article in this volume. The equivalence of nuclearity and local completely positive factorization of the identity map through matrix algebras is due to Choi and Effros \cite{choi78} and Kirchberg \cite{kirchberg77} and was extended in the appropriate way to the operator system category in \cite{kptt2013}.
\end{remark}

\begin{remark} {}\
    \begin{enumerate}
        \item It is trivial that $M_n$ is nuclear for each $n=1,2,\dotsc$ and that $M_n(E)$ is nuclear if and only if $E$ is.
        
        \item The classes of AF and UHF C$^*$-algebras as described in Szabo's article in this volume are nuclear. The Cuntz algebras $\cc O_n$ defined therein are also nuclear.
        
        \item $\cc B(H)$ is not nuclear for $H$  infinite dimensional: see \cite[Proposition 2.4.9]{BrownOzawa}.
        
        \item It is false in general that C$^*$-subalgebras of nuclear C$^*$-algebras are nuclear. 
    \end{enumerate}
    
\end{remark}

\begin{remark}
    There is a subsystem of $M_3$ which is not nuclear. Indeed, let $E$ be the subsystem of all matrices $A\in M_3$ with $A_{13} = A_{31} =0$. If $E$ was nuclear there would be sequences of unital completely positive maps $\vp_n: E\to M_n$ and $\psi_n: M_n\to E$ so that $\psi_n\circ\vp_n\to \id_E$. By Arveson's extension theorem, we may extend each $\vp_n$ to a unital completely positive map $\tilde\vp_n: M_3\to M_n$, and we see that a cluster point of $\psi_n\circ\tilde\vp_n$ gives a completely positive projection $p: M_3\to E$. By \cite[Theorem 3.1]{Choi1977}, this would define a C$^*$-algebra structure on $E$ under the multiplication $x\cdot y: = p(xy)$. There are two $\ast$-homomorphisms $M_2\to E$,
    \[A\mapsto \begin{bmatrix} A & 0\\ 0 & 0\end{bmatrix}, \begin{bmatrix} 0 & 0\\ 0 & A\end{bmatrix}\]
    with distinct ranges, which are respected by the C$^*$-algebra structure on $E$. However, since every finite-dimensional C$^*$-algebra is a direct sum of matrix algebras this is a contradiction since $E$ is seven-dimensional but contains two distinct (non-unital) $\ast$-subalgebras isomorphic to $M_2$, so would need to have dimension at least eight.
\end{remark}

\begin{exercise}
      Let $X$ be a compact, Hausdorff space. Show that $C(X)$, the algebra of continuous functions $f:X\to\bb C$ under the norm $\|f\|_{\infty} := \sup_x |f(x)|$ is nuclear. As a hint, use the fact that for every finite open cover $\cc O =\{O_1,\dots,O_n\}$ of $X$ there is a \emph{partition of unity} subordinate to $\cc O$. That is, there are continuous functions $f_i:X\to [0,1]$ with $f_i$ supported in $O_i$ and $\sum_{i=1}^n f_i(x) =1$ for all $x\in X$.
\end{exercise}

With the definition and basic examples established, we know turn to understanding the model theory of nuclear operator systems.

\begin{defn}
    Let $T$ be a theory of $\cc L$-structures, and let $x = (x_1,x_2,\dotsc,x_n)$ and $y=(y_1,y_2,\dotsc,y_n)$ be finite collections of symbols. Let $\cc F$ be a family of (definable) formulas for $T$, each depending only on $x$. We say that $\cc F$ is \emph{uniform} if there is a continuous function $u: \bb R^+\to \bb R^+$ with $u(0)=0$ so that
    \[ T\models |f(x) - f(y)|\leq u(\max_i d(x_i,y_i)),\ \textup{for all}\ f\in \cc F.\]
    
    For a class $\cc C$ of $\cc L$-structures so that $C \models T$ for all $C\in \cc C$, we say that $\cc C$ is \emph{definable by uniform families of formulas} if there is a collection of formulas $\{f_{n,k} : n,k=1,2,\dotsc\}$ defined on a countable number of symbols $x=(x_1,x_2,\dotsc)$ so that $\cc F_k := \{f_{n,k}: n=1,2,\dotsc\}$ is uniform for each $k$ and
    \[[E\models T\ \textup{and}\ \sup_{k}\inf_{n} f_{n,k}(x)^E \equiv 0]\ \textup{if and only if}\ E\in \cc C.\]
    
    We refer \cite[Section 5.7]{model-c-star} for the connection between definability by uniform families of formulas and being definable by omitting (partial) types.
\end{defn}

\begin{prop} \label{nuclear-omitting-types}
    The property that an operator system $E$ is nuclear is definable by uniform families of existential formulas.
\end{prop}

\begin{proof}
    The formulas $f_{n,k}$ are given by 
    \begin{equation*}
        f_{n,k}(x_1,\dotsc,x_k) := \inf\left\{\max_i\|x_i - \psi\circ\vp(x_i)\| : \vp\in \ucp(E, M_n),\ \psi\in \ucp(M_n,E)\right\}
    \end{equation*}
    which are existential and definable by Propositions \ref{prop:ucp-mn-e-def} and \ref{prop:state-def} and Corollary \ref{cor:ucp-e-mn-def}. Since $\psi\circ\vp$ is a contraction, we have that each $f_{n,k}$ is $2$-Lipschitz, so the families $\cc F_k$ are uniform. Finally, if $\inf_n f_{n,k}(x)^E= 0$ for all tuples $x$ in $E$, we can easily construct nets $(\vp_i)$ and $(\psi_i)$ indexed over the directed set of pairs $(F,\e)$ where $F$ is a finite subset of $E$ and $\e>0$ so that $\psi_i\circ\vp_i$ converges to the identity on $E$ in the pointwise-norm topology, so $E$ is nuclear. Conversely, the existence of such a net easily implies that $\inf_n f_{n,k}^E\equiv 0$ for all $k$.
\end{proof}

\begin{question}
    For the elementary class of unital C$^*$-algebras, is nuclearity defined by a uniform family of \emph{positive} existential formulas?
\end{question}

Since the values of positive existential formulas decrease under surjective morphisms, this would recover the result due to Choi and Effros, \cite[Corollary 3.3]{choi-effros-duke} and \cite[Corollary 4]{choi-effros-iu}, that nuclear C$^*$-algebras are closed under taking C$^*$-algebra quotients avoiding Connes' \cite{Connes-inj} celebrated but difficult work on the classification of injective factors. However, as noted in Remark \ref{rmk:nuclear-quotient}, this seems like it would require substantially different ideas to avoid the false implication in the operator system category. We refer to \cite[Section 5.9]{model-c-star} for one such attempt.

\begin{defn}
    Let $E$ be a finite-dimensional operator system. We say that $E$ is \emph{exact} if for every $\e>0$ there is a unital, completely positive embedding $\vp: E\to M_n$ for some $n$ so that $\|\vp^{-1}|_{\vp(E)}\|_{\cb}\leq 1+\e$. We say that an operator system is \emph{exact} if every finite-dimensional subsystem is exact.
\end{defn}

\begin{remark} {}\
    \begin{enumerate}
        \item Every nuclear operator system is exact.
        
        \item It is clear from the definition that any subsystem of an exact operator system is exact. As a special case, every subsystem of $M_n$ is exact.
        
        \item Since every subsystem of any unital abelian $C^*$-algebra A is exact, $\MIN(V)$ is exact for any archimedean ordered $\ast$-vector space $V$.
        
        \item A difficult theorem of Kirchberg \cite{kirchberg-uhf} establishes that C$^*$-algebraic quotients of exact C$^*$-algebras are exact. This is again false in the operator system category by the same counterexample as given in Remark \ref{rmk:nuclear-quotient}.
    \end{enumerate}
\end{remark}

\begin{remark}
    In Vignati's article in this volume a finite-dimensional \emph{operator space}, that is, a closed subspace of $\cc B(H)$, is said to be $c$-exact for some $c\geq 1$ if for every $\e>0$ there is a completely contractive embedding $\vp: E\to M_n$ for some $n$ so that $\|\vp^{-1}|_{\vp(E)}\|_{\cb}\leq c+\e$.
    
    Let $E\subset \cc B(H)$ be a finite-dimensional operator space. Define an operator system $\tilde E\subset M_2(\cc B(H))$ by 
    \[ \tilde E := \left\{\begin{bmatrix} \la 1 & x\\ y^* & \mu 1\end{bmatrix} : x,y\in E,\ \la,\mu\in \bb C\right\}.\]
    We have that $E$ is $1$-exact as an operator space if and only if $\tilde E$ is an exact operator system. We leave this as an exercise based on the following ``$2\times 2$-matrix trick'' due to Paulsen: $\vp: E\to \cc B(K)$ is completely contractive if and only if $\tilde\vp: \tilde E\to M_2(\cc B(K))$ given by \[\tilde\vp\left(\begin{bmatrix} \la 1 & x\\ y^* & \mu 1\end{bmatrix}\right) := \begin{bmatrix} \la 1 & \vp(x)\\ \vp(y)^* & \mu 1\end{bmatrix}\] is (unital) completely positive. We refer the reader to \cite[Theorem B.5]{BrownOzawa} or \cite[Lemma 8.1]{paulsen2002completely} for a proof of this fact.
\end{remark}

We point out that exactness, similarly to nuclearity, can be phrased as in terms of local completely positive factorization of some (equivalently, every) \emph{inclusion} $E\subset \cc B(H)$ through matrix algebras, the crucial difference being that the images of the approximating sequence need not land back in the system itself.

\begin{prop} \label{prop:exact-cp}
    A finite-dimensional operator system $E\subset \cc B(H)$ is exact if and only if for each $\e>0$ there are unital, completely positive maps $\vp: E\to M_n$ and $\psi: \vp(E)\to \cc B(H)$ so that $\|\psi\circ\vp - \id_E\|< \e$ where $\id_E$ is the restriction of the identity map on $\cc B(H)$ to $E$.
\end{prop}

\begin{proof}
    We have that $\vp(E)\subset M_n$ and that $\vp^{-1}: \vp(E) \to \cc B(H)$ is unital and self-adjoint. By \cite[Corollary B.9]{BrownOzawa} there is a unital, completely positive map $\psi:\vp(E)\to \cc B(H)$ with $\|\psi-\vp^{-1}\|_{\cb}\leq 2(\|\vp^{-1}\|_{\cb}-1)$; hence, 
    \[\|\psi\circ\vp - \id_E\|\leq \|\psi\circ\vp - \id_E\|_{\cb} =\|\psi\circ\vp - \vp^{-1}\circ\vp\|_{\cb} \leq \|\psi-\vp^{-1}\|_{\cb}\|\vp\|_{\cb}\leq 2\e.\]
    
    In the other direction,  suppose there are unital, completely positive maps $\vp,\psi$ as above, and let $y_1,\dotsc,y_n,y_1^*,\dotsc,y_n^*$ be such a basis/dual basis pair for $\vp(E)$ as given in Lemma \ref{lem:os-dual-basis}. We have that
    \[(\psi - \vp^{-1})(z) = \sum_{i=1}^n y_i^*(z)(\psi(y_i) - \vp^{-1}(y_i)),\]
    so
    \[\|\psi - \vp^{-1}\|\leq \dim(E)\max_{i}\|\psi(y_i) - \vp^{-1}(y_i)\|.\]
    Setting $y_i = \vp(x_i)$ we have that \[\|x_i\|\leq \|\psi\circ\vp(x_i) - x_i\| + \|\psi\circ\vp(x_i)\|\leq \|\vp(x_i)\|+\e =1+\e;\] 
    hence, $\max_i \|\psi(y_i) - \vp^{-1}(y_i)\| = \max_i \|\psi\circ\vp(x_i) - x_i\|\leq (1+\e)\e$.
\end{proof}

Given the similarities between nuclearity and exactness it is natural to wonder if exact operator systems have a similar low-complexity description. A result of Goldbring and the author \cite{gs-omitting} shows that this is, however, not the case.

\begin{prop}
    The class of exact operator systems is not definable by uniform families of existential formulas.
\end{prop}

\begin{question}
    Is the class of exact C$^*$-algebras definable by uniform families of (positive, existential) formulas?
\end{question}

\subsection{The Lifting Property}

\begin{defn}
    An operator system $E$ has the \emph{lifting property} (LP) if for every unital C$^*$-algebra $A$ and every ideal $I$, for every unital, completely positive map $\vp: E\to A/I$ there exists a unital, completely positive map $\tilde\vp: E\to A$ lifting $\vp$, that is, $q\circ \tilde\vp = \vp$ where $q: A\to A/I$ is the quotient map.
\end{defn}

The following result is due to Robertson and Smith \cite{RobertsonSmith}.

\begin{prop} \label{prop:robertson-smith}
    Let $E$ be a finite-dimensional operator system and $B/J$ be a quotient C$^*$-algebra. For each $n=1,2,\dotsc$ every unital, completely positive map $\vp: E\to B/J$ admits a unital $n$-positive lifting $\tilde\vp: E\to B$.
\end{prop}

The following is immediate from Definition \ref{defn:maximal}.

\begin{cor}
    Every finite-dimensional, $k$-maximal operator system has the lifting property. 
\end{cor}

\begin{defn}
    A finite-dimensional operator system $E$ is said to be \emph{CP-stable} if for every $\e>0$ there exist $\delta>0$ and $n$ so that for any unital map $\vp: E\to A$ into any C$^*$-algebra $A$ satisfying $\|\vp\|_n\leq 1 + \de$, there is a unital, completely positive map $\vp': E\to A$ so that $\|\vp' - \vp\|\leq \e$. 
\end{defn}

Parts of the following result are due to Kavruk \cite[Theorem 6.6]{Kavruk2014} and Goldbring and the author \cite[Proposition 2.42]{gs-kirchberg}, \cite[Section 7]{gs-omitting}, as well as \cite{sinclair-cp}.

\begin{thm} \label{thm:lp-cp-stable}
     For a finite-dimensional operator system $E$, the following are equivalent:
     \begin{enumerate}
         \item $E$ has the lifting property;
         \item $E^*$ is exact;
         \item $E$ is CP-stable;
         \item for every sequence $(A_i)$ of unital C$^*$-algebras and every non-principal ultrafilter $\cc U$, every unital, completely positive map $\vp: E\to \prod_{\cc U} A_i$ admits a unital, completely positive lifting $\tilde\vp: E\to \prod A_i$;
         \item for every sequence $(M_{n_i})$ of matrix algebras and every non-principal ultrafilter, $\cc U$ every unital, completely positive map $\vp: E\to \prod_{\cc U} M_{n_i}$ admits a unital, completely positive lifting $\tilde\vp: E\to \prod M_{n_i}$.
     \end{enumerate}
\end{thm}

\begin{remark}
    Let $E$ be a finite dimensional operator system with a hermitian basis $\{1=x_0,x_1,\dotsc,x_n\}$, and consider the elementary class $\cc C$ of all operator systems which are unital C$^*$-algebras \cite{gs-axiom}. Consider the uniform assignment on $\cc C$ given by \[L_E(A) := \{(a_1,\dotsc,a_n)\in (A^h)^n : \vp(x_i) = a_i,\ i=1,\dotsc,n,\ \textup{for some}\ \vp\in \ucp(E,A)\}.\] Then the equivalence $(1)\Leftrightarrow (4)$ shows that $E$ has the lifting property if and only if $L_E$ is definable.
\end{remark}

Theorem \ref{thm:lp-cp-stable} allows one to give an interesting interpretation of the lifting property for certain maximal-type operator systems in terms of definability via Proposition \ref{prop:lmi-beth}. 

\begin{example} \label{ex:free-op-sys}
    Let $p(x)$ with $x=(x_1,\dots,x_n)$ be a homogeneous, hermitian linear matrix $\ast$-polynomial of degree $d$. Consider the closed convex set 
        \[ K := \{x\in \bb R^n : 1_d\succeq p(x).\}\]
    (Any convex set of this form is known as a \emph{spectrahedron}.) Suppose that $K$ is bounded with $0$ as an interior point. Recalling the construction from Remark \ref{rmk:cone} of the archimedean ordered $\ast$-vector space $V_K = (\bb C^{n+1}, P_K, e)$ from $K$, we define an operator system structure on $V_K$ by defining $K_\ell = \ucp(V_K,M_\ell)$ to be given by all linear maps $f:\bb R^{n+1}\to M_\ell^h$ with $f(e_1) = I_{\ell}$ so that the tuple $X = (X_1,\dotsc,X_n)$ given by $X_i := f(e_{i+1})$ satisfies
        \[1_d\otimes 1_\ell \succeq p(X)\]
    where $e_1,e_2,\dotsc,e_{n+1}$ form the standard basis for $\bb R^{n+1}$. (Note the order unit $e$ is $e_1$.)
    The collection of sets $(K_\ell)_{\ell\in\bb N}$ is known as the \emph{free spectrahedron} defined by $p(x)$ \cite{kriel}. Consequently, we refer to the operator system structure we have just constructed on $V_K$ as the \emph{free operator system} generated by $p$, which we will denote by $\cc F(p)$.
\end{example}

\begin{prop}
    Using the same assumptions and notation as Example \ref{ex:free-op-sys}, the free operator system $\cc F(p)$ has the lifting property.
\end{prop}

\begin{proof}
    Let $\vp: \cc F(p)\to \cc B(H)$ be a unital, completely positive map, and define $X_i := \vp(e_{i+1})$ for $i=1,\dotsc,n$ as above. Since $\cc F(p)$ is finite-dimensional, we may assume without loss of generality that $H$ is separable. (This is just to avoid the slight inconvenience of working with nets rather than sequences, but the subsequent arguments translate to nets practically verbatim.)  Let $H_1\subset H_2\subset \dotsb H$ be an increasing sequence of finite-dimensional subspaces whose union is dense in $H$, and let $p_j: H\to H_j$ be the orthogonal projection onto $H_j$. We have that $\vp^{(j)}: \cc F(p)\to M_{m_j}$ defined by $\vp^{(j)}(x) := p_j \vp(x) p_j$ is unital, completely positive; thus, for the tuple $X^{(j)} = (X_1^{(j)},\dotsc, X_n^{(j)})$ given by $X_i^{(j)} := \vp^{(j)}(e_{i+1}) = p_j \vp(e_{i+1}) p_j$ we have that $X^{(j)}$ is hermitian and $1_d\otimes 1_{m_j}\succeq p(X^{(j)})$ holds. Since $X^{(j)}$ converges to $X$ in the strong operator topology, $p(X^{(j)})$ also converges strongly to $p(X)$, so $1_d\otimes 1_H\succeq p(X)$ as well.
    
    Now let $\vp: \cc F(p)\to \prod_{\cc U} M_{m_j}$ be a unital, completely positive map. Since $\cc F(p)$ is finite-dimensional, we may assume that $\vp = (\vp_j)_\cc U$ where $\vp_j: \cc F(p)\to M_{m_j}$ is unital, $\ast$-linear, and sends $\bb R^{n+1}$ into $M_{m_j}^h$. Setting $Y_i^{(j)} := \vp_j(e_{i+1})$, we have that 
    \[\dist(1_d\otimes 1_{m_j} - p(Y^{(j)}), M_{dm_j}^+)<\e\ \textup{for}\ j\in \cc U\ \textup{generic.}\] 
    By Proposition \ref{prop:lmi-beth}, we may perturb the tuple $Y^{(j)}$ slightly to $Z^{(j)}$ so that $\vp_j'(e_{i+1}) := Z_i^{(j)}$ defines a unital completely positive map. We have that $\tilde\vp = (\vp_j'): \cc F(p)\to \prod M_{m_j}$ is the required lifting. By the equivalence of (1) and (5) in Theorem \ref{thm:lp-cp-stable}, the result obtains. \qedhere
\end{proof}

\begin{remark}
    Let $q(x)$ be a homogenous, hermitian linear operator $\ast$-polynomial, and assume that the closed, convex set $K = \{x\in \bb R^n : 1_H\succeq q(x)\}$ is bounded and contains $0$ as an interior point. Then, we may construct an operator system $\cc F(q)$ in exactly the same way as we did in Example \ref{ex:free-op-sys}, which we will still refer to as the \emph{free operator system} generated by $q$. The proof of the previous proposition applies equally well to this situation, using now Proposition \ref{prop:lmi-beth-op}, so $\cc F(q)$ is seen to have the lifting property. Since every closed, convex set $K$ arises in this way by combining all linear inequalities it satisfies into a single linear operator $\ast$-polynomial, this recovers that all maximal operator systems are CP-stable. (The equivalence of CP-stability and the lifting property in Theorem \ref{thm:lp-cp-stable} itself uses Proposition \ref{prop:robertson-smith}.)
\end{remark}

\subsection{The Weak Expectation Property}

\begin{defn} \label{defn:wep}
    We say that an operator system $E$ has the \emph{weak expectation property} if for every unital inclusion $E\subset B$ into another operator system, there is a unital, completely positive map $\vp: B\to E^{**}$ so that $\vp(x) = \hat x$ for all $x\in E$. (Recall that $\hat\cdot : E\to E^{**}$ is the evaluation map as in Remark \ref{rem:archimedean} and Proposition \ref{prop:E-dual-dual}.)
\end{defn}

\begin{lem} \label{lem:weak-wep}
    An operator system $E$ has the weak expectation property if and only if for every unital inclusion $E\subset B$ and each hermitian $b\in B^h$, there is a unital, completely positive map $\vp: E+\bb Cb\to E^{**}$ so that $\vp(x)=\hat x$ for all $x\in E$.
\end{lem}

\begin{proof}
    Let us fix $E\subset B$. We first show that for every finite set $b_1,\dotsc,b_n\in B$ there is a unital, completely positive map $\vp: E+\bb Cb_1 + \dotsb + \bb Cb_n\to E^{**}$ so that $\vp(x) = \hat x$ for all $x\in E$. Inductively, assume that this is true for choice of $n$ elements, and consider $b_1,\dotsc,b_{n+1}\in B^h$. Let $\vp: E+\bb Cb_1 + \dotsb + \bb Cb_n\to E^{**}$ be such a map. We have that $E^{**}\subset \cc B(H)$ for some Hilbert space $H$, so by Arveson's extension theorem, we may extend $\vp$ to a unital, completely positive map $\vp': B\to \cc B(H)$. We have that $E\cong \widehat E\subset \cc B(H)$, so setting $d := \vp'(b_{n+1})$ by assumption there is a unital, completely positive map $\psi: E+ \bb Cd\to E^{**}$ so that $\psi(x) = \hat x$ for all $x\in E$. Let $\psi': (E+\bb Cd)^{**}\cong E^{**}+\bb Cd\to E^{**}$ be the continuous extension of $\psi$ in the weak* topology, which is again unital and completely positive. Since $\vp'(E+\bb Cb_1 + \dotsb + \bb Cb_{n+1})\subset E^{**}+\bb Cd$, we have that $\psi'\circ\vp': E+\bb Cb_1 + \dotsb + \bb Cb_{n+1}\to E^{**}$ is the required map.
    
    Since every $x\in B$ is a linear combination of two hermitian elements, we have demonstrated that for every subsystem $E\subset A\subset B$ with $A$ having finite co-dimension relative to $E$, that there is a unital, completely positive map $\vp_A: A\to E^{**}$ so that $\vp(x)=\hat x$ for all $x\in E$. Consider the directed set $\cc F$ of all such subsystems of $B$, ordered by inclusion. We may take a pointwise-weak* cluster point of the net $(\vp_A)_{A\in \cc F}$ to define a unital, completely positive map $\vp: B\to E^{**}$ so that $\vp(x)=\hat x$ for all $x\in E$. This verifies that $E$ has the weak expectation property.
\end{proof}

\begin{lem} \label{lem:wep-finitary}
    An operator system $E$ has the weak expectation property if and only if for every unital inclusion $E\subset B$ of operator systems for each finite-dimensional subsystem $A\subset B$, $n=1,2,\dotsc$, and $\e>0$, there is a $n$-contractive $\ast$-linear map $\vp: A\to E$ so that $\|\vp(x) - x\|\leq \e\|x\|$ for all $x\in E\cap A$. Moreover, it suffices to check that the second statement holds only for all $A = A_0 + \bb C b$ where $A_0\subset E$ is a finite-dimensional subsystem and $b\in B^h$.
\end{lem}

\begin{proof}
    By the proof of Proposition \ref{prop:E-dual-dual}, we have that $M_n(E^{**})$ is completely order isomorphic to $M_n(E)^{**}$; hence, if $\vp: B\to E^{**}$ is a unital, completely positive map witnessing the weak expectation property, by the principle of local reflexivity \cite[Theorem 12.2.4]{albiac} for any finite dimensional subsystem $A\subset B$ and $\e>0$ there is a contraction $\theta: M_n(A)\to M_n(E)$ so that $\|\theta(x) - \vp_n(x)\|<\e\|x\|$ for all $x\in M_n(E\cap A)$. Let $\cc G$ be the group of all signed permutation matrices and set
    \[\theta'(x) = \frac{1}{|\cc G|^2} \sum_{g,h\in \cc G} g^*\theta(gxh)h^*.\] Since $g^*\vp_n(gxh)h^* = \vp_n(x)$ for all $g,h\in \cc G$, we have that $\|\theta'(x) - \vp_n(x)\|<\e\|x\|$ for all $x\in M_n(E\cap A)$ as well. It is not hard to check that $\theta'(e_{ij}\otimes x) = e_{ij}\otimes \psi(x)$, $i,j=1,\dotsc,n$, for a unique linear map $\psi:A\to E$, so that $\psi$ is $n$-contractive. Since $\psi$ is linear and $n$-contractive, so is $\psi^*(x) := \psi(x^*)^*$; thus, the average of the two is $n$-contractive and $\ast$-linear, with the other properties being preserved.

    In the other direction, let $\cc F$ bet the directed set of all finite-dimensional subsystems of $B$ ordered by inclusion, and let $I$ be the directed set of triples $(A,n,\e)$ where $A\in \cc F$, $n$ is a positive integer, and $\e>0$, ordered by $(A',n',\e')\geq (A,n,\e)$ if $A'\supset A$, $n'\geq n$, and $\e'\leq \e$. For each $(A,n,\e)\in I$, find $\vp_{A,n,\e}: A\to E$ matching the conditions of the second statement. Equip $E^{**}$ with the weak* topology, and take $\vp: B\to E^{**}$ to be a pointwise-weak* cluster point of the net $(\vp_{A,n,\e})$, where we identify $E$ with $\widehat E\subset E^{**}$. We have that $\vp$ is unital, $\ast$-linear, and completely contractive with $\vp(x) = \hat x$ for all $x\in E$. By Proposition \ref{prop:contraction-positive}, $\vp$ is also completely positive.
    
    For the last part, the previous argument can be easily adapted to verify that for each hermitian $b\in B^h$, there is a unital, completely positive map $\vp: E+\bb Cb\to E^{**}$ so that $\vp(x)=\hat x$ for all $x\in E$. Thus, the result follows by Lemma \ref{lem:weak-wep}.
\end{proof}

Recall from Definition \ref{defn:positive} that a formula $f$ is positive if it is built from atomic formulas only using connectives $\theta(x_1,\dotsc,x_n)$ satisfying $\theta(x_1,\dotsc,x_n)\leq \theta(y_1,\dotsc,y_n)$ when $x_i\leq y_i$, $i=1,\dotsc, n$.

\begin{defn}
    We say that an operator system $E$ is \emph{positively existentially closed} if for every unital inclusion $E\subset B$ of operator systems and every positive, quantifier-free formula $f(x,y)$ we have that $\inf_y f(x,y)^E = \inf_y f(x,y)^B$ for all tuples $x$ chosen from $E$.
\end{defn}

\begin{remark} \label{rmk:pec-map}
    We can as easily talk about a specific inclusion $E\subset B$ being \emph{positively existential} if the condition in the definition is met. In the operator system category, this is equivalent to the existence of an ultrafilter $\cc U$ on some set $I$ so that there is a unital, completely positive map $\vp: B\to E^{\cc U}$ so that $\vp(x) = \iota(x)$ for all $x\in E$, where $\iota: E\to E^{\cc U}$ is the constant embedding, cf. \cite[Remark 4.20]{barlak-szabo}. The proof of Lemma \ref{lem:wep-finitary} thus shows the following.
    
    \begin{prop} \label{prop:dual-dual-pec}
        The canonical inclusion $\hat\cdot: E\to E^{**}$ is positively existential for any operator system $E$.
    \end{prop}
\end{remark}

\begin{remark} \label{rmk:pec}
    Since $\inf_y f(x,y)^E \geq\inf_y f(x,y)^B$ automatically, checking positive existential closure is equivalent to checking $\inf_y f(x,y)^E \leq \inf_y f(x,y)^B$ where the tuple $x$ is chosen from $E$. It is straightforward to verify that this in turn is equivalent to checking that for any finite collection $f_1(x,y),\dotsc,f_k(x,y)$ of atomic formulas (that is, $f_i(x,y) = \|p_i(x,y)\|_{d_i}$ for $p_i(x,y)$ a linear matrix $\ast$-polynomial of degree $d_i$), any $\e>0$, and any tuples $a$ in $E$ and $b$ in $B$, there is a tuple $b'$ in $E$ so that $f_i(a,b')\leq f_i(a,b)+\e$ for all $i=1,\dotsc,k$.
    
\end{remark}

We say that an element $x\in \cc B(H)$ is \emph{strictly positive} if $x - \e1$ is positive for some $\e>0$. We write $x\succ 0$ to denote that $x$ is strictly positive.

\begin{lem} \label{lem:wep-positive}
    An operator system is positively existentially closed if and only if for any unital inclusion $E\subset B$ and any finite collection $p_1(x,y),\dotsc, p_k(x,y)$ of hermitian linear matrix $\ast$-polynomials, the following property holds:
    
    \begin{itemize}
        \item[(R)] for any tuples $a$ in $E$ and $b$ in $B$ so that \[p_1(a,b)\succ 0, \dotsc, p_k(a,b) \succ 0,\] there is a tuple $b'$ in $E$ so that \[p_1(a,b')\succ 0, \dotsc, p_k(a,b')\succ 0.\]
    \end{itemize}
\end{lem}

\begin{proof}
    Assume that (R) holds. Let $p_1(x,y),\dotsc,p_k(x,y)$ be a collection of linear matrix $\ast$-polynomials of degree $d_1,\dotsc,d_k$. Fixing tuples $a$ in $E$ and $b$ in $B$ and $\e>0$, define $r_i := \|p_i(a,b)\|_{d_i}$ and consider the hermitian, linear matrix $\ast$-polynomials $q_1(x,y),\dotsc,q_k(x,y)$ where 
    \[q_i(x,y) := \begin{bmatrix} (r_i+\e) 1_{d_i} & p_i(x,y)\\ p_i^*(x,y) & (r_i + \e)1_{d_i}\end{bmatrix}.\]
    By Proposition \ref{prop:order-norm}, we have that $q_i(a,b)\succ 0$ for all $i=1,\dotsc,k$. Therefore, by (R) there exists a tuple $b'$ in $E$ so that $q_i(a,b')\succ 0$ for all $i=1,\dotsc,k$, and again by Proposition \ref{prop:order-norm} we have that $\|p_i(a,b')\|_{d_i}< r_i + \e$ for $i=1,\dotsc,k$. That $E$ is positively existentially closed now follows from Remark \ref{rmk:pec}.
    
    For the converse, let $p_1(x,y),\dotsc,p_k(x,y)$ be a collection of hermitian linear matrix $\ast$-polynomials of degrees $d_1,\dotsc,d_k$, and fix tuples $a$ in $E$ and $b$ in $B$ so that $p_i(a,b)\succeq \e 1_{d_i}$ for all $i=1,\dotsc,k$ for some $\e>0$ sufficiently small. Setting $r_i = \|p_i(a,b)\|_{d_i}$, we define 
    \[q_i(x,y) := 2p_i(x,y) - (r_i+\e)1_{d_i}.\]
    By some basic arithmetic we have that $-(r_i-\e)1_{d_i}\preceq q_i(a,b)\preceq (r_i-\e)1_{d_i}$, which is equivalent to $\|q_i(a,b)\|_{d_i}\leq r_i-\e$. Thus $E$ being positively existentially closed guarantees that there is a tuple $b'$ in $E$ so that $\|q_i(a,b')\|_{d_i}\leq r_i$ for all $i=1,\dotsc,k$. Since $q_i(x,y)$ is hermitian it only takes values in the hermitian elements; thus, \[-r_i1_{d_i}\preceq q_i(a,b') = 2p_i(a,b') - r_i1_{d_i} - \e 1_{d_i}\] for all $i=1,\dotsc,k$ which shows that $p_i(a,b')\succ 0$ for $i=1,\dotsc,k$.
\end{proof}

Since a positivity check on a collection of linear matrix $\ast$-polynomials is equivalent to a positivity check on a single linear matrix $\ast$-polynomial obtained from the collection via block diagonal embedding, the following is immediate from the proof of the previous result.

\begin{cor}
     An operator system is positively existentially closed if and only if for any unital inclusion $E\subset B$ and any hermitian linear matrix $\ast$-polynomial $p(x,y)$ of degree $d$, one of the following equivalent properties holds:
    
    \begin{itemize}
        \item[(R$_1$)] for any tuples $a$ in $E$ and $b$ in $B$ so that $p(a,b)\succ 0$ there is a tuple $b'$ in $E$ so that $p(a,b')\succ 0$.
        \item[(R$_2$)] for any tuples $a$ in $E$ and $b$ in $B$ so that $\|p(a,b)\|_d< 1$ there is a tuple $b'$ in $E$ so that $\|p(a,b')\|< 1$.
    \end{itemize}
\end{cor}

The following result was proved by Goldbring and the author, \cite[Section 2.4]{gs-kirchberg} and \cite[Proposition 5.1]{gs-omitting}, and by Lupini \cite[Proposition 3.2]{lupini-wep2018} in full generality. We note that in light of Remark \ref{rmk:pec-map} a characterization of the weak expectation property very much in the same spirit appears as \cite[Theorem 9.22]{pisier-ck}.

\begin{prop} \label{prop:wep-equals-pec}
    An operator system has the weak expectation property if and only if it is positively existentially closed.
\end{prop}

\begin{proof}
    Assume that $E$ has the weak expectation property and that $E\subset B$ is an inclusion of operator systems. Let $p(x,y)$ be a hermitian linear matrix $\ast$-polynomial. Suppose there are tuples $a=(a_1,\dotsc,a_n)$ in $E$ and $b=(b_1,\dotsc,b_k)$ in $B$ so that $\|p(a,b)\|_d< 1$, and let $A\subset B$ be the subsystem spanned by $a_1,\dotsc,a_n,b_1,\dotsc,b_k$. As guaranteed by Lemma \ref{lem:wep-finitary}, choose $\vp: A\to E$ to be $\ast$-linear, $d$-contractive, and have the property that $\max_i\|\vp(a_i) - a_i\|<\e$ for $\e>0$ suitably small, to be determined later. We have that
    \[\|p(a,\vp(b))\|_d \leq \|p(\vp(a),\vp(b))\|_d + L\e = \|\vp(p(a,b))\|_d + L\e\leq \|p(a,b)\|_d + L\e,\]
    where $L$ is the Lipschitz constant for the map $x\mapsto\|p(x,\vp(b))\|_d$. We now choose $\e>0$ so that $\|p(a,b)\|_d + L\e<1$.
    
    In the other direction, it is clear that for any tuples $a$ in $E$, $b$ in $B$, and $c$ in $E^{**}$ that the statement, ``the $\ast$-linear map $\vp$ on the span of $a$ and $b$ induced by sending $a$ to $a$ and $b$ to $c$ is $n$-contractive'' can be verified by the (countably) infinite collection of expressions of the form $\|p_i(a,b)\|_{n}< r_i\Rightarrow \|p_i(a,c)\|_{n}<r_i$ for $p_i$ linear matrix $\ast$-polynomials with complex rational coefficients; thus, fixing $a$ and $b$, by taking weak*-limits of tuples $c_m$ in $E = \widehat E\subset E^{**}$ satisfying the first $m$ statements we can define an $n$-contractive $\ast$-linear map $\vp_{a,b}$ from the span of $a$ and $b$ to $E^{**}$ satisfying $\vp(a_i) = \hat a_i$. The required map $\vp: B\to E^{**}$ is then obtained by taking a pointwise-weak* limit of the $\vp_{a,b}$'s over the directed set of all finite tuples.
\end{proof}

In light of Lemmas \ref{lem:wep-finitary} and \ref{lem:wep-positive}, the proof of the previous proposition can be easily adapted to prove the following.

\begin{cor} \label{cor:pec-single}
     An operator system is positively existentially closed if and only if for any unital inclusion $E\subset B$ and any hermitian linear matrix $\ast$-polynomial $p(x,y) = p(x_1,\dotsc,x_n,y)$ of degree $d$, one of the following equivalent properties holds:
    
    \begin{itemize}
        \item[(R$_3$)] for any tuple $a$ in $E$ and $b\in B^h$ so that $p(a,b)\succ 0$ there is $b'\in E^h$ so that $p(a,b')\succ 0$.
        \item[(R$_4$)] for any tuple $a$ in $E$ and $b\in B^h$ so that $\|p(a,b)\|_d< 1$ there is $b'\in E^h$ so that $\|p(a,b')\|< 1$.
    \end{itemize}
\end{cor}

\begin{question}
    Is there a proof of the previous corollary which uses positive existential closure directly without passing through the equivalence with the weak expectation property?
\end{question}

\begin{remark} For the (elementary) class of operator systems which are C$^*$-algebras, there is an explicit, countable collection of existentially quantified linear matrix $\ast$-polynomial inequalities which verifies the weak expectation property due to Farenick, Paulsen, and Kavruk \cite[Theorem 6.1]{fkp-wep}. Consider the hermitian linear matrix $\ast$-polynomial
\[ p(x_1,x_2,y_1,y_2) := \begin{bmatrix} y_1 & x_1 & 0\\ x_1^* & y_2 & x_2\\ 0 & x_2^* & 1 - y_1 -y_2 \end{bmatrix}.\]
A unital C$^*$-algebra $A$ has the weak expectation property if and only if for every unital inclusion $A\subset B$ of operator systems and every $n=1,2,\dotsc$ whenever $a_1,a_2\in M_n(A)$ and $b_1,b_2\in M_n(B)$ are so that $p(a_1,a_2,b_1,b_2)\succ 0$, there are $b_1',b_2'\in M_n(A)$ so that $p(a_1,a_2,b_1',b_2')\succ 0$.

\begin{remark}
    Note that since strict inequalities are used in the preceding, the existential conditions verifying the weak expectation property given are not first-order expressible. By replacing the entries of $p$ in the Farenick--Kavruk--Paulsen result with matrices, we actually increase the number of variables as each variable is replaced by a collection of variables standing in for the matrix units. 
    \begin{question}
        Is there a countable sequence $p_i(x,y)$ of hermitian linear matrix $\ast$-polynomials in a \emph{uniformly bounded} number of variables so that a C$^*$-algebra has the weak expectation property if and only if for all embeddings $A\subset \cc B(H)$ and all tuples $a$ in $A$ and $b$ in $\cc B(H)$ so that $p_i(a,b)\succ 0$ for all $i=1,2,\dotsc$ there is a tuple $b'$ in $A$ so that $p_i(a,b')\succ 0$? Would a single $p$ suffice?
    \end{question}
\end{remark}

\end{remark}

\begin{remark}
    Perhaps the simplest non-trivial hermitian linear matrix $\ast$-polynomial is
    \[ p(x,y) = \begin{bmatrix} 1 - y & x^*\\ x & y \end{bmatrix}.\] For any C$^*$-algebra $A\subset \cc B(H)$ it is an exercise in functional calculus to show that if for any $a\in A$ there is some $b\in \cc B(H)$ so that $p(a,b)\succ 0$, then there is a $b'\in A$ so that $p(a,b')\succ 0$. (It is a bit easier to see that there is $b'\in A^{**}$ since $A^{**}$ is a von Neumann algebra, then use Proposition \ref{prop:dual-dual-pec}.) As is pointed out in the introduction of \cite{fkp-wep}, we have that $\exists b: p(a,b)\succ 0$ is equivalent to the \emph{numerical radius} $\omega(a)$ of $a$ being less than $1/2$. It would be interesting to know if the weak expectation property could be characterized in terms of this simple expression.
    \begin{question}
        With $p$ as in the previous remark, does a C$^*$-algebra $A$ have the weak expectation property if for some (for all) $n$ for all $a_1,\dotsc,a_n\in A$ for which there exists $b\in \cc B(H)$ so that $p(a_i,b)\succ 0$ for all $i=1,\dotsc,n$ there is $b'\in A$ so that $p(a_i,b')\succ 0$ for all $i=1,\dotsc,n$?
    \end{question}
\end{remark}

We close this section with a model-theoretic proof of a characterization of the weak expectation property which is an unpublished result due to Kavruk \cite{kavruk-riesz}. Kavruk's proof uses the theory of operator system tensor products \cite{KPTT2011,kptt2013} along with Kirchberg's tensorial characterization of the weak expectation property \cite{kirchberg1993}.

\begin{notation}
    In the following, we write $\bowtie$ to indicate a fixed choice of $\prec$ or $\succ$, so that, for instance, $a_1\bowtie b_1$, $a_2\bowtie b_2$ means one of four possible expressions with one choice of $\prec$ or $\succ$ for each pair $a_i,b_i$ which remains fixed throughout.
\end{notation}

\begin{defn}
    We say that any operator system $E$ has the \emph{complete tight Riesz interpolation property} if for any unital inclusion of operator systems $E\subset B$, any $n=1,2,...$, and any tuple $a= (a_1\dotsc,a_k)$ in $M_n(E)^h$ and $b\in M_n(B)^h$ if $b \bowtie a_i$ for all $i=1,\dotsc,k$, then there is a $b'\in M_n(E)^h$ so that $b' \bowtie a_i$ for all $i=1,\dotsc,k$. 
\end{defn}

\begin{thm}[Kavruk {\cite[Theorem 7.4]{kavruk-riesz}}]
    An operator system $E$ has the complete tight Riesz interpolation property if and only if it has the weak expectation property.
\end{thm}

\begin{proof}
    Proposition \ref{prop:wep-equals-pec} and Lemma \ref{lem:wep-positive} show that the weak expectation property implies the complete tight Riesz interpolation property.
    
    For the converse, we will show that that the complete tight Riesz interpolation property implies property (R$_3$), therefore the weak expectation property by Proposition \ref{prop:wep-equals-pec} and Corollary \ref{cor:pec-single}. Let $E\subset B$ be a unital inclusion of operator systems and $p(x_1,\dotsc,x_n,y)$ be a hermitian linear matrix $\ast$-polynomial of degree $d$. Assume we have chosen $a_1,\dotsc,a_n\in E$ and $b\in B^h$ so that $p(a_1,\dotsc,a_n,b)\succ 0$. Since $b$ is hermitian, we have that 
    \[ p(a_1,\dotsc,a_n,b) = X\otimes b - A \succ 0\]
    where $X\in M_d^h$ and $A\in M_n(E)^h$. By perturbing $X$ by $\e I_d$ for $\e$ sufficiently small, we may assume without loss of generality that $X$ is invertible. Therefore, we have that $X = Y\Sg Y^*$ for $Y$ invertible and $\Sg$ a diagonal matrix with $\pm 1$ entries on the main diagonal. We may assume the first $f$ entries of the main diagonal of $\Sg$ are $1$ with the remaining being $-1$. Let $\cc G$ be the group of signed permutation matrices in $M_d$ preserving the set $\{\pm e_1, \dots, \pm e_f\}$ where $e_1,\dotsc,e_d$ is the standard basis for $\bb C^d$.
    
    For each $g\in \cc G$, define $A_g := gY^{-1}A(Y^{-1})^*g^*\in M_d(E)^h$, and for $z\in B^h$ consider the system of inequalities
    \[\{\Sg\otimes z \succ A_g : g\in \cc G\},\]
    noting that this system is consistent if and only if $X\otimes z\succ A$ since $g(\Sg\otimes b)g^* = \Sg\otimes b$ for all $g\in \cc G$. By the complete tight Riesz interpolation property, there is $b'\in M_d(E)^h$ so that $b' \succ A_g$ for all $g\in \cc G$. We have that $gb'g^* \succ Y^{-1}A(Y^{-1})^*$ for all $g\in \cc G$; thus,
    \[b'' := \frac{1}{|\cc G|} \sum_{g\in \cc G} gb'g^* \succ Y^{-1}A(Y^{-1})^*.\]
    It is not hard to check that $b''$ is block diagonal of the form
    \[ b'' = \begin{bmatrix} I_f\otimes b_1 & 0\\ 0 & I_{d-f}\otimes b_2\end{bmatrix}\]
    for some $b_1,b_2\in E^h$. It now suffices to find $c\in E^h$ with $c\succ b_1-\e1$ and $-c\succ b_2-\e1$ for $\e>0$ sufficiently small as this would imply \[\Sg\otimes c \succ b'' -\e 1_d \succ Y^{-1}A(Y^{-1})^*,\ \textup{therefore}\ X\otimes c\succ A.\] 
    
     In order to do so, we choose $\e>0$ witnessing the strict positivity of $b''$ as in the right-hand side of the preceding equation line and set $\de = \e/2$. Realizing $E\subset \cc B(H)$ and setting $b_i^\delta = b_i -\de 1$, we have that $b_i^\de  = [b_i^\de]_+ - [b_i^\de]_{-}$ for some $[b_i^\de]_+,[b_i^\de]_-\in \cc B(H)^+$, $i=1,2$. Setting $c' := [b_2^\de]_+ - [b_1^\de]_+\in \cc B(H)^h$, we can check that
    \[b_1 - \e1 \prec b_1 - \de1 \preceq c'\preceq -b_2 + \de 1 \prec -b_2 + \e1.\]
    Hence by the complete tight Riesz interpolation property there is $c\in E^h$ so that $b_1 - \e1 \prec c \prec -(b_2 -\e 1)$ verifying the claim. \qedhere
    
\end{proof}

\section*{Acknowledgements} The author was supported by NSF grant DMS-2055155. The author thanks Isaac Goldbring, Connor Thompson, and the anonymous referee for suggesting many corrections and improvements.


\begin{bibdiv}
\begin{biblist}

\bib{albiac}{book}{
   author={Albiac, Fernando},
   author={Kalton, Nigel J.},
   title={Topics in Banach space theory},
   series={Graduate Texts in Mathematics},
   volume={233},
   publisher={Springer, New York},
   date={2006},
   pages={xii+373},
   isbn={978-0387-28141-4},
   isbn={0-387-28141-X},
   review={\MR{2192298}},
}

\bib{alfsen}{book}{
   author={Alfsen, Erik M.},
   title={Compact convex sets and boundary integrals},
   series={Ergebnisse der Mathematik und ihrer Grenzgebiete, Band 57},
   publisher={Springer-Verlag, New York-Heidelberg},
   date={1971},
   pages={x+210},
   review={\MR{0445271}},
}

\bib{aubrun}{book}{
   author={Aubrun, Guillaume},
   author={Szarek, Stanis\l aw J.},
   title={Alice and Bob meet Banach},
   series={Mathematical Surveys and Monographs},
   volume={223},
   note={The interface of asymptotic geometric analysis and quantum
   information theory},
   publisher={American Mathematical Society, Providence, RI},
   date={2017},
   pages={xxi+414},
   isbn={978-1-4704-3468-7},
   review={\MR{3699754}},
   doi={10.1090/surv/223},
}

\bib{barlak-szabo}{article}{
   author={Barlak, Sel\c{c}uk},
   author={Szab\'{o}, G\'{a}bor},
   title={Sequentially split $\ast$-homomorphisms between $\rm
   C^*$-algebras},
   journal={Internat. J. Math.},
   volume={27},
   date={2016},
   number={13},
   pages={1650105, 48},
   issn={0129-167X},
   review={\MR{3589655}},
   doi={10.1142/S0129167X16501056},
}


\bib{model-metric}{article}{
   author={Ben Yaacov, Ita\"{\i}},
   author={Berenstein, Alexander},
   author={Henson, C. Ward},
   author={Usvyatsov, Alexander},
   title={Model theory for metric structures},
   conference={
      title={Model theory with applications to algebra and analysis. Vol. 2},
   },
   book={
      series={London Math. Soc. Lecture Note Ser.},
      volume={350},
      publisher={Cambridge Univ. Press, Cambridge},
   },
   date={2008},
   pages={315--427},
   review={\MR{2436146}},
   doi={10.1017/CBO9780511735219.011},
}









\bib{BrownOzawa}{book}{
   author={Brown, Nathanial P.},
   author={Ozawa, Narutaka},
   title={$C^*$-algebras and finite-dimensional approximations},
   series={Graduate Studies in Mathematics},
   volume={88},
   publisher={American Mathematical Society, Providence, RI},
   date={2008},
   pages={xvi+509},
   isbn={978-0-8218-4381-9},
   isbn={0-8218-4381-8},
   review={\MR{2391387}},
   doi={10.1090/gsm/088},
}

\bib{Choi1975}{article}{
    AUTHOR = {Choi, Man Duen},
     TITLE = {Completely positive linear maps on complex matrices},
   JOURNAL = {Linear Algebra Appl.},
  FJOURNAL = {Linear Algebra and its Applications},
    VOLUME = {10},
      YEAR = {1975},
     PAGES = {285--290},
      ISSN = {0024-3795},
   MRCLASS = {15A60 (46L05)},
  MRNUMBER = {376726},
MRREVIEWER = {E. St\o rmer},
       DOI = {10.1016/0024-3795(75)90075-0},
       URL = {https://doi.org/10.1016/0024-3795(75)90075-0},
}

\bib{choi-effros-duke}{article}{
   author={Choi, Man Duen},
   author={Effros, Edward G.},
   title={Separable nuclear $C\sp*$-algebras and injectivity},
   journal={Duke Math. J.},
   volume={43},
   date={1976},
   number={2},
   pages={309--322},
   issn={0012-7094},
   review={\MR{405117}},
}

\bib{Choi1977}{article}{
  author={Choi, Man Duen},
  author={Effros, Edward G.},
  title={Injectivity and operator spaces},
  journal={Journal of functional analysis},
  volume={24},
  number={2},
  pages={156--209},
  year={1977},
  publisher={Elsevier},
}

\bib{choi-effros-iu}{article}{
   author={Choi, Man Duen},
   author={Effros, Edward G.},
   title={Nuclear $C\sp*$-algebras and injectivity: the general case},
   journal={Indiana Univ. Math. J.},
   volume={26},
   date={1977},
   number={3},
   pages={443--446},
   issn={0022-2518},
   review={\MR{430794}},
   doi={10.1512/iumj.1977.26.26034},
}

\bib{choi78}{article}{
   author={Choi, Man Duen},
   author={Effros, Edward G.},
   title={Nuclear $C\sp*$-algebras and the approximation property},
   journal={Amer. J. Math.},
   volume={100},
   date={1978},
   number={1},
   pages={61--79},
   issn={0002-9327},
   review={\MR{482238}},
   doi={10.2307/2373876},
}

\bib{Connes-inj}{article}{
   author={Connes, A.},
   title={Classification of injective factors. Cases $II_{1},$
   $II_{\infty },$ $III_{\lambda },$ $\lambda \not=1$},
   journal={Ann. of Math. (2)},
   volume={104},
   date={1976},
   number={1},
   pages={73--115},
   issn={0003-486X},
   review={\MR{454659}},
   doi={10.2307/1971057},
}

\bib{Conway}{book}{
   author={Conway, John B.},
   title={A course in functional analysis},
   series={Graduate Texts in Mathematics},
   volume={96},
   edition={2},
   publisher={Springer-Verlag, New York},
   date={1990},
   pages={xvi+399},
   isbn={0-387-97245-5},
   review={\MR{1070713}},
}

\bib{effros-nc-order}{article}{
   author={Effros, Edward G.},
   title={Aspects of noncommutative order},
   conference={
      title={${\rm C}^{\ast}$-algebras and applications to physics},
      address={Proc. Second Japan-USA Sem., Los Angeles, Calif.},
      date={1977},
   },
   book={
      series={Lecture Notes in Math.},
      volume={650},
      publisher={Springer, Berlin},
   },
   date={1978},
   pages={1--40},
   review={\MR{504750}},
}

\bib{model-c-star}{article}{
   author={Farah, Ilijas},
   author={Hart, Bradd},
   author={Lupini, Martino},
   author={Robert, Leonel},
   author={Tikuisis, Aaron},
   author={Vignati, Alessandro},
   author={Winter, Wilhelm},
   title={Model theory of $\rm C^*$-algebras},
   journal={Mem. Amer. Math. Soc.},
   volume={271},
   date={2021},
   number={1324},
   pages={viii+127},
   issn={0065-9266},
   isbn={978-1-4704-4757-1},
   isbn={978-1-4704-6626-8},
   review={\MR{4279915}},
   doi={10.1090/memo/1324},
}

\bib{fhs-ii}{article}{
   author={Farah, Ilijas},
   author={Hart, Bradd},
   author={Sherman, David},
   title={Model theory of operator algebras II: model theory},
   journal={Israel J. Math.},
   volume={201},
   date={2014},
   number={1},
   pages={477--505},
   issn={0021-2172},
   review={\MR{3265292}},
   doi={10.1007/s11856-014-1046-7},
}

\bib{Farenick2012}{article}{
   author={Farenick, Douglas},
   author={Paulsen, Vern I.},
   title={Operator system quotients of matrix algebras and their tensor
   products},
   journal={Math. Scand.},
   volume={111},
   date={2012},
   number={2},
   pages={210--243},
   issn={0025-5521},
   review={\MR{3023524}},
   doi={10.7146/math.scand.a-15225},
}

\bib{fkp-wep}{article}{
   author={Farenick, Douglas},
   author={Kavruk, Ali S.},
   author={Paulsen, Vern I.},
   title={$C^\ast$-algebras with the weak expectation property and a
   multivariable analogue of Ando's theorem on the numerical radius},
   journal={J. Operator Theory},
   volume={70},
   date={2013},
   number={2},
   pages={573--590},
   issn={0379-4024},
   review={\MR{3138370}},
   doi={10.7900/jot.2011oct07.1938},
}



\bib{GartnerMatousek2012}{book}{
   author={G\"{a}rtner, Bernd},
   author={Matou\v{s}ek, Ji\v{r}\'{\i}},
   title={Approximation algorithms and semidefinite programming},
   publisher={Springer, Heidelberg},
   date={2012},
   pages={xii+251},
   isbn={978-3-642-22014-2},
   isbn={978-3-642-22015-9},
   review={\MR{3015090}},
   doi={10.1007/978-3-642-22015-9},
}


\bib{gs-kirchberg}{article}{
   author={Goldbring, Isaac},
   author={Sinclair, Thomas},
   title={On Kirchberg's embedding problem},
   journal={J. Funct. Anal.},
   volume={269},
   date={2015},
   number={1},
   pages={155--198},
   issn={0022-1236},
   review={\MR{3345606}},
   doi={10.1016/j.jfa.2015.02.016},
}

\bib{gs-omitting}{article}{
   author={Goldbring, Isaac},
   author={Sinclair, Thomas},
   title={Omitting types in operator systems},
   journal={Indiana Univ. Math. J.},
   volume={66},
   date={2017},
   number={3},
   pages={821--844},
   issn={0022-2518},
   review={\MR{3663327}},
   doi={10.1512/iumj.2017.66.6019},
}

\bib{gs-axiom}{article}{
   author={Goldbring, Isaac},
   author={Sinclair, Thomas},
   title={On the axiomatizability of $\rm C^*$-algebras as operator systems},
   journal={Glasg. Math. J.},
   volume={61},
   date={2019},
   number={3},
   pages={629--635},
   issn={0017-0895},
   review={\MR{3991361}},
   doi={10.1017/s001708951800040x},
}

\bib{grothendieck}{article}{
   author={Grothendieck, Alexandre},
   title={Produits tensoriels topologiques et espaces nucl\'{e}aires},
   language={French},
   journal={Mem. Amer. Math. Soc.},
   volume={16},
   date={1955},
   pages={Chapter 1: 196 pp.; Chapter 2: 140},
   issn={0065-9266},
   review={\MR{75539}},
}


\bib{kadison}{article}{
   author={Kadison, Richard V.},
   title={A representation theory for commutative topological algebra},
   journal={Mem. Amer. Math. Soc.},
   volume={7},
   date={1951},
   pages={39},
   issn={0065-9266},
   review={\MR{44040}},
}

\bib{kavruk-riesz}{article}{
       author = {Kavruk, Ali S.},
        title = {The Weak Expectation Property and Riesz Interpolation},
      journal = {arXiv e-prints},
     keywords = {Mathematics - Operator Algebras},
         year = {2012},
          eid = {arXiv:1201.5414},
        pages = {arXiv:1201.5414},
archivePrefix = {arXiv},
       eprint = {1201.5414},
 primaryClass = {math.OA},
       adsurl = {https://ui.adsabs.harvard.edu/abs/2012arXiv1201.5414K},
      adsnote = {Provided by the SAO/NASA Astrophysics Data System}
}

\bib{Kavruk2014}{article}{
   author={Kavruk, Ali S.},
   title={Nuclearity related properties in operator systems},
   journal={J. Operator Theory},
   volume={71},
   date={2014},
   number={1},
   pages={95--156},
   issn={0379-4024},
   review={\MR{3173055}},
   doi={10.7900/jot.2011nov16.1977},
}

\bib{Kavruk2015}{article}{
   author={Kavruk, Ali S.},
   title={On a non-commutative analogue of a classical result of Namioka and
   Phelps},
   journal={J. Funct. Anal.},
   volume={269},
   date={2015},
   number={10},
   pages={3282--3303},
   issn={0022-1236},
   review={\MR{3401618}},
   doi={10.1016/j.jfa.2015.09.002},
}

\bib{KPTT2011}{article}{
   author={Kavruk, Ali S.},
   author={Paulsen, Vern I.},
   author={Todorov, Ivan G.},
   author={Tomforde, Mark},
   title={Tensor products of operator systems},
   journal={J. Funct. Anal.},
   volume={261},
   date={2011},
   number={2},
   pages={267--299},
   issn={0022-1236},
   review={\MR{2793115}},
   doi={10.1016/j.jfa.2011.03.014},
}

\bib{kptt2013}{article}{
   author={Kavruk, Ali S.},
   author={Paulsen, Vern I.},
   author={Todorov, Ivan G.},
   author={Tomforde, Mark},
   title={Quotients, exactness, and nuclearity in the operator system
   category},
   journal={Adv. Math.},
   volume={235},
   date={2013},
   pages={321--360},
   issn={0001-8708},
   review={\MR{3010061}},
   doi={10.1016/j.aim.2012.05.025},
}

\bib{kirchberg77}{article}{
   author={Kirchberg, Eberhard},
   title={$C\sp*$-nuclearity implies CPAP},
   journal={Math. Nachr.},
   volume={76},
   date={1977},
   pages={203--212},
   issn={0025-584X},
   review={\MR{512362}},
   doi={10.1002/mana.19770760115},
}

\bib{kirchberg1993}{article}{
   author={Kirchberg, Eberhard},
   title={On nonsemisplit extensions, tensor products and exactness of group
   $C^*$-algebras},
   journal={Invent. Math.},
   volume={112},
   date={1993},
   number={3},
   pages={449--489},
   issn={0020-9910},
   review={\MR{1218321}},
   doi={10.1007/BF01232444},
}

\bib{kirchberg-uhf}{article}{
   author={Kirchberg, Eberhard},
   title={Commutants of unitaries in UHF algebras and functorial properties
   of exactness},
   journal={J. Reine Angew. Math.},
   volume={452},
   date={1994},
   pages={39--77},
   issn={0075-4102},
   review={\MR{1282196}},
   doi={10.1515/crll.1994.452.39},
}

\bib{kriel}{article}{
   author={Kriel, Tom-Lukas},
   title={An introduction to matrix convex sets and free spectrahedra},
   journal={Complex Anal. Oper. Theory},
   volume={13},
   date={2019},
   number={7},
   pages={3251--3335},
   issn={1661-8254},
   review={\MR{4020034}},
   doi={10.1007/s11785-019-00937-8},
}

\bib{Lance1995}{book}{
   author={Lance, E. C.},
   title={Hilbert $C^*$-modules},
   series={London Mathematical Society Lecture Note Series},
   volume={210},
   note={A toolkit for operator algebraists},
   publisher={Cambridge University Press, Cambridge},
   date={1995},
   pages={x+130},
   isbn={0-521-47910-X},
   review={\MR{1325694}},
   doi={10.1017/CBO9780511526206},
}

\bib{Lovasz2003}{article}{
   author={Lov\'{a}sz, L.},
   title={Semidefinite programs and combinatorial optimization},
   conference={
      title={Recent advances in algorithms and combinatorics},
   },
   book={
      series={CMS Books Math./Ouvrages Math. SMC},
      volume={11},
      publisher={Springer, New York},
   },
   date={2003},
   pages={137--194},
   review={\MR{1952986}},
   doi={10.1007/0-387-22444-0\_6},
}



\bib{lupini-wep2018}{article}{
   author={Lupini, Martino},
   title={An intrinsic order-theoretic characterization of the weak
   expectation property},
   journal={Integral Equations Operator Theory},
   volume={90},
   date={2018},
   number={5},
   pages={Paper No. 55, 17},
   issn={0378-620X},
   review={\MR{3829543}},
   doi={10.1007/s00020-018-2479-x},
}



   \bib{paulsen2002completely}{book}{
    AUTHOR = {Paulsen, Vern},
     TITLE = {Completely bounded maps and operator algebras},
    SERIES = {Cambridge Studies in Advanced Mathematics},
    VOLUME = {78},
 PUBLISHER = {Cambridge University Press, Cambridge},
      YEAR = {2002},
     PAGES = {xii+300},
      ISBN = {0-521-81669-6},
}

\bib{ptt2011}{article}{
   author={Paulsen, Vern I.},
   author={Todorov, Ivan G.},
   author={Tomforde, Mark},
   title={Operator system structures on ordered spaces},
   journal={Proc. Lond. Math. Soc. (3)},
   volume={102},
   date={2011},
   number={1},
   pages={25--49},
   issn={0024-6115},
   review={\MR{2747723}},
   doi={10.1112/plms/pdq011},
}

\bib{paulsen-tomforde}{article}{
   author={Paulsen, Vern I.},
   author={Tomforde, Mark},
   title={Vector spaces with an order unit},
   journal={Indiana Univ. Math. J.},
   volume={58},
   date={2009},
   number={3},
   pages={1319--1359},
   issn={0022-2518},
   review={\MR{2542089}},
   doi={10.1512/iumj.2009.58.3518},
}



\bib{PisierIntro}{book}{
   author={Pisier, Gilles},
   title={Introduction to operator space theory},
   series={London Mathematical Society Lecture Note Series},
   volume={294},
   publisher={Cambridge University Press, Cambridge},
   date={2003},
   pages={viii+478},
   isbn={0-521-81165-1},
   review={\MR{2006539}},
   doi={10.1017/CBO9781107360235},
}

\bib{pisier-ck}{book}{
   author={Pisier, Gilles},
   title={Tensor products of $C^*$-algebras and operator spaces---the
   Connes-Kirchberg problem},
   series={London Mathematical Society Student Texts},
   volume={96},
   publisher={Cambridge University Press, Cambridge},
   date={2020},
   pages={x+484},
   isbn={978-1-108-74911-4},
   isbn={978-1-108-47901-1},
   review={\MR{4283471}},
   doi={10.1017/9781108782081},
}

\bib{RobertsonSmith}{article}{
   author={Robertson, A. G.},
   author={Smith, R. R.},
   title={Liftings and extensions of maps on $C^*$-algebras},
   journal={J. Operator Theory},
   volume={21},
   date={1989},
   number={1},
   pages={117--131},
   issn={0379-4024},
   review={\MR{1002124}},
}

\bib{Ruan1978}{thesis}{
    author={Ruan, Zhong-Jin},
  title={On matricially normed spaces associated with operator algebras},
  year={1987},
  school={University of California, Los Angeles}
}

\bib{sinclair-cp}{article}{
   author={Sinclair, Thomas},
   title={CP-stability and the local lifting property},
   journal={New York J. Math.},
   volume={23},
   date={2017},
   pages={739--747},
   review={\MR{3665586}},
}

\bib{takesaki-i}{book}{
   author={Takesaki, M.},
   title={Theory of operator algebras. I},
   series={Encyclopaedia of Mathematical Sciences},
   volume={124},
   note={Reprint of the first (1979) edition;
   Operator Algebras and Non-commutative Geometry, 5},
   publisher={Springer-Verlag, Berlin},
   date={2002},
   pages={xx+415},
   isbn={3-540-42248-X},
   review={\MR{1873025}},
}

\bib{webster-winkler}{article}{
   author={Webster, Corran},
   author={Winkler, Soren},
   title={The Krein-Milman theorem in operator convexity},
   journal={Trans. Amer. Math. Soc.},
   volume={351},
   date={1999},
   number={1},
   pages={307--322},
   issn={0002-9947},
   review={\MR{1615970}},
   doi={10.1090/S0002-9947-99-02364-8},
}

\bib{Xhabli-thesis}{book}{
   author={Xhabli, Blerina},
   title={Universal operator system structures on ordered spaces and their
   applications},
   note={Thesis (Ph.D.)--University of Houston},
   publisher={ProQuest LLC, Ann Arbor, MI},
   date={2009},
   pages={122},
   isbn={978-1109-70014-5},
   review={\MR{2736708}},
}

\bib{Xhabli-jfa}{article}{
   author={Xhabli, Blerina},
   title={The super operator system structures and their applications in
   quantum entanglement theory},
   journal={J. Funct. Anal.},
   volume={262},
   date={2012},
   number={4},
   pages={1466--1497},
   issn={0022-1236},
   review={\MR{2873847}},
   doi={10.1016/j.jfa.2011.11.009},
}


\end{biblist}
\end{bibdiv}
\end{document}